\documentclass{amsart}
\usepackage[utf8]{inputenc}
\usepackage{amsmath}
\usepackage{amsfonts}
\usepackage{amssymb}
\usepackage{amsthm}
\usepackage{amscd}
\usepackage{float}
\usepackage{tikz}
\usepackage{graphicx}
\usepackage[colorlinks=true]{hyperref}
\hypersetup{urlcolor=blue, citecolor=red}
\usepackage{hyperref}

\newtheorem{theorem}{Theorem}[section]

\newtheorem{lemma}[theorem]{Lemma}
\newtheorem{proposition}[theorem]{Proposition}
\newtheorem{conjecture}[theorem]{Conjecture}

\theoremstyle{definition}
\newtheorem{definition}[theorem]{Definition}
\newtheorem{remark}[theorem]{Remark}

\newcommand{\R}{\mathbb{R}}
\newcommand{\N}{\mathbb{N}}
\newcommand{\C}{\mathbb{C}}
\newcommand{\T}{\mathbb{T}}
\newcommand{\Z}{\mathbb{Z}}

\renewcommand{\a}{\mathbf{a}}
\renewcommand{\b}{\mathbf{b}}

\begin{document}

\title{Local smoothing for the Hermite wave equation}

\author{Robert Schippa}

\address{Current address: UC Berkeley, Department of Mathematics, 970 Evans Hall
Berkeley, CA 94720-3840}
\email{robert.schippa@gmail.com}

\begin{abstract}
We show local smoothing estimates in $L^p$-spaces for solutions to the Hermite wave equation. For this purpose, we obtain a parametrix given by a Fourier Integral Operator, which we linearize. This leads us to analyze local smoothing estimates for solutions to Klein-Gordon equations. We show $\ell^2$-decoupling estimates adapted to the mass parameter to obtain local smoothing with essentially sharp derivative loss. In one dimension as consequence of square function estimates, we obtain estimates sharp up to endpoints. Finally, we elaborate on the implications of local smoothing estimates for Hermite Bochner--Riesz means.

 \end{abstract}

\maketitle

\section{Introduction}

In the following we consider local smoothing estimates for solutions to the wave equation with Harmonic potential, which is referred to as Hermite wave equation:
\begin{equation}
\label{eq:WaveHermiteEquation}
\left\{ \begin{array}{cl}
\partial_t^2 u &= \Delta u - |x|^2 u, \quad (x,t) \in \R^d \times \R, \\
u(0) &= u_0 \in L^p_s(\R^d), \quad \dot{u}(0) = 0.
\end{array} \right.
\end{equation}
Above $d \geq 1$, $s \geq 0$, and $L^p_s(\R^d)$
denotes the $L^p$-based Sobolev space of order $s$. We denote with $\mathcal{H} = - \Delta + |x|^2$ the Hermite operator, which is presently regarded as an unbounded operator on $L^2(\R^d)$. It has discrete spectrum and its eigenfunctions are the explicitly known Hermite eigenfunctions.

\medskip

 By local smoothing estimates for solutions $u$ to \eqref{eq:WaveHermiteEquation} we mean space-time $L^p$-estimates
\begin{equation}
\label{eq:LocalSmoothingWaveHermiteIntroduction}
\| u \|_{L^p_{t,x}([0,1] \times \R^d)} \lesssim \| u_0 \|_{L^p_s(\R^d)}.
\end{equation}

\medskip

The study of local smoothing estimates for solutions to the Euclidean wave equation was initiated by Sogge \cite{Sogge1991}. For context we recall that the fixed-time $L^p$-estimates for solutions to the Euclidean wave equation
\begin{equation}
\label{eq:EuclideanWaveIntroduction}
\left\{ \begin{array}{cl}
\partial_t^2 v &= \Delta v, \quad (x,t) \in \R^d \times \R, \\
v(0) &= v_0, \quad \dot{v}(0) = 0
\end{array} \right.
\end{equation}
read due to Peral and Miyachi \cite{Peral1980,Miyachi1980}
\begin{equation*}
\| v(t) \|_{L^p_x(\R^d)} \lesssim \| v(0) \|_{L^p_{s_p}(\R^d)}
\end{equation*}
with $s_p = (d-1) \big| \frac{1}{2} - \frac{1}{p} \big|$ for $p \in (1,\infty)$.

\medskip

Sogge observed that Bourgain's proof of the circular maximal theorem in two dimensions \cite{Bourgain1986} implies improved space-time $L^p$-estimates for solutions to \eqref{eq:EuclideanWaveIntroduction}:
\begin{equation*}
\| v \|_{L^p_{t,x}([0,1] \times \R^2)} \lesssim \| v(0) \|_{L^p_{s_p-\varepsilon}(\R^2)} \text{ for } 2 < p < \infty.
\end{equation*}
It is conjectured that
\begin{equation*}
\| v \|_{L^p_{t,x}([0,1] \times \R^d)} \lesssim \| v(0) \|_{L^p_{s_p - \sigma_p}(\R^d)}
\end{equation*}
for $\frac{2d}{d-1} < p < \infty$ with $\sigma_p < \frac{1}{p}.$ For a recent exhaustive review we refer to \cite{BeltranHickmanSogge2021}.

\medskip

The local smoothing conjecture is known to imply as well the Bochner-Riesz as the restriction conjecture. Both are still open in dimensions greater than $2$. Hence, the proof of sharp local smoothing estimates is an extremely difficult problem and has only been accomplished for $d=2$ by Guth--Wang--Zhang \cite{GuthWangZhang2020}. In higher dimensions the currently best range for sharp local smoothing estimates follow from $\ell^2$-decoupling due to Bourgain--Demeter \cite{BourgainDemeter2015}; see references therein for previous progress. Decoupling estimates were conceived by Wolff \cite{Wolff2000} (see also \cite{LabaWolff2002}) to make progress on the local smoothing conjecture. Let
\begin{equation*}
\mathcal{E}_c f(x,t)= \int_{\R^d} e^{i(x \cdot \xi + t |\xi|)} f(\xi) \chi_1(\xi) d\xi
\end{equation*}
denote the Fourier extension operator for the cone; $\chi_1$ denotes a smooth version of the indicator function of the unit annulus. The decoupling estimates at the critical index $p = \frac{2(d+1)}{d-1}$ read
\begin{equation}
\label{eq:ConeDecouplingIntroduciton}
\| \mathcal{E}_c f \|_{L^p_{t,x}(B_{d+1}(0,R))} \lesssim_\varepsilon R^\varepsilon \big( \sum_{\theta: R^{-\frac{1}{2}}-\text{sector}} \| \mathcal{E}_c f_{\theta} \|_{L^p_{t,x}(w_{B_{d+1}(0,R)})}^2 \big)^{\frac{1}{2}}.
\end{equation}
Above $w_{B_{d+1}(0,R)}$ denotes a weight adapted to $B_{d+1}(0,R)$ with high polynomial decay (say $100d$) off $B_{d+1}(0,R)$. 
This implies local smoothing estimates for the Euclidean wave equation for $p \geq \frac{2(d+1)}{d-1}$ with derivative loss sharp up to endpoints. Guth--Wang--Zhang \cite{GuthWangZhang2020} proved the local smoothing conjecture for $d=2$ by means of a sharp square function estimate combined with a maximal function estimate; the latter being more classical \cite{MockenhauptSeegerSogge1992,Fefferman1973,
Cordoba1977}. For recent progress in higher dimensions beyond $\ell^2$-decoupling (though not with the sharp order of derivatives) we refer to \cite{GaoLiuMiaoXi2023,BeltranSaari2022}.

\smallskip

We digress for a moment to compare local smoothing estimates for the Hermite wave equation to smoothing estimates for solutions to wave equations on compact Riemannian manifolds $(M,g)$:
\begin{equation}
\label{eq:RiemannianWaveEquation}
\left\{ \begin{array}{cl}
\partial_t^2 u &= \Delta_g u, \quad (t,x) \in \R \times M, \\
u(0) &= u_0 \in L^p_s(M), \quad \dot{u}(0) = 0.
\end{array} \right.
\end{equation}
One approach to local smoothing estimates for solutions to \eqref{eq:RiemannianWaveEquation} is based on the Fourier Integral Operator (FIO) representation of the parametrix for the half-wave equation:
\begin{equation}
\label{eq:LocalSolutionRiemannianWaveEquation}
u(t,x) = \int e^{i \phi(x,t;\xi)} a(x,t;\xi) \hat{u}_0(\xi) d\xi + R_t u,
\end{equation}
where $\phi$ solves the eikonal equation with $p(x,\xi)= |\xi|_{g(x)} = \sqrt{g^{ij}(x) \xi_i \xi_j}$. Here $g$ denotes the Riemannian metric in local coordinates, and $R_t$ an $L^p$-bounded remainder term.

\smallskip

Beltran--Hickman--Sogge \cite{BeltranHickmanSogge2020} showed that the decoupling estimates \eqref{eq:ConeDecouplingIntroduciton} transpire to the FIO described by \eqref{eq:LocalSolutionRiemannianWaveEquation}; see also \cite[Chapter~7]{Schippa2019PhD} in the context of variable-coefficient Schr\"odinger propagation. The proof leans on induction-on-scales, which is facilitated by the self-similar structure of the decoupling estimates. On sufficiently small scales, the FIO admits a pointwise approximation by its linearization. The linearization can be decoupled via the Bourgain--Demeter \cite{BourgainDemeter2015} result. For further progress on local smoothing estimates for wave equations on Riemannian manifolds see \cite{GaoLiuMiaoXi20232d,Schippa2023}.
It is important to note that in case of Riemannian manifolds $(M,g)$, the range of $L^p$-spaces, for which local smoothing estimates are expected, is strictly smaller for $\dim(M) \geq 3$ than in the Euclidean case as pointed out in examples by Minicozzi--Sogge \cite{MinicozziSogge1997}.

\medskip

The parametrix for solutions to the Hermite half-wave equation is given by
\begin{equation}
\label{eq:ParametrixHermiteHalfwaveEquation}
u(t,x) = \int e^{i \phi(x,t;\xi)} a(x,t,\xi) \hat{u}_0(\xi) d\xi,
\end{equation}
where $\phi$ solves the eikonal equation with principal symbol $p(x,\xi) = \sqrt{|x|^2 + |\xi|^2}$:
\begin{equation*}
\left\{ \begin{array}{cl}
\partial_t \phi(x,t;\xi) &= p(x,\nabla_x \phi), \\
\phi(x,0;\xi) &= \langle x, \xi \rangle.
\end{array} \right.
\end{equation*}
It suffices to show local smoothing estimates for $u$ given by \eqref{eq:ParametrixHermiteHalfwaveEquation} under the phase space localization $\{ |x-x_0| \lesssim 1 \}$ and $\{ |\xi| \sim N \}$. Indeed, the Hamilton flow is given by
\begin{equation}
\label{eq:HamiltonFlow}
\left\{ \begin{array}{cl}
\dot{x}_t &= \frac{\xi}{\sqrt{|x|^2 + |\xi|^2}}, \\
\dot{\xi}_t &= -\frac{x}{\sqrt{|x|^2 + |\xi|^2}}.
\end{array} \right.
\end{equation}
Note that $|x|^2 + |\xi|^2$ is a conserved quantity. Moreover, \eqref{eq:HamiltonFlow} indicates finite speed of propagation. When we consider a compact time interval, we can localize the solution in space on the same scale. Secondly, the frequencies are not changing essentially. If the frequencies are initially localized at scale $N$, at any point $x \in \R^d$, $|\dot{\xi}_t| \lesssim 1$. This means that the frequencies do not essentially change their modulus nor their directions, which corresponds to a difference in angle of size $N^{-\frac{1}{2}}$; we refer to \cite[Section~5.2]{HassellPortalRozendaalYung2023} for related considerations.

We remark that in one dimension Greiner--Holcman--Kannai \cite{GreinerHolcmanKannai2002} obtained an explicit solution formula for
\begin{equation*}
\left\{ \begin{array}{cl}
\partial_t^2 u &= \partial_x^2 u - |x|^2 u, \quad (t,x) \in \R \times \R, \\
u(0) &= 0, \quad \dot{u}(0) = \delta_0.
\end{array} \right.
\end{equation*}
Above $\delta_0 \in \mathcal{S}'(\R)$ denotes the $\delta$-distribution at the origin. The representation in \cite[Proposition~7]{GreinerHolcmanKannai2002} involves a contour integral and establishes a sharp form of finite speed of propagation. Kannai \cite{Kannai2000}  obtained another less implicit expression of the fundamental solution. We remark that in dimensions $d \geq 2$ d'Ancona--Pierfelice--Ricci \cite{DanconaPierfeliceRicci2010} proved the dispersive properties of the Euclidean wave equation on the unit time scale by Dirichlet series estimates.
However, for the following considerations a parametrix given by Fourier Integral Operators seems most tractable, and we hope that the present construction can be used for related problems.

\medskip

Compare \eqref{eq:HamiltonFlow} to the Hamilton flow associated to the half-wave equation on a compact Riemannian manifold with principal symbol $p(x,\xi) = |\xi|_{g(x)} = \sqrt{g^{ij}(x) \xi_i \xi_j}$:
\begin{equation*}
\left\{ \begin{array}{cl}
\dot{x}_t &= \frac{g^{-1} \xi}{|\xi|_{g(x)}}, \\
\dot{\xi}_t &= \frac{(\partial_x g^{ij}) \xi_i \xi_j}{|\xi|_{g(x)}}.
\end{array} \right.
\end{equation*}
We still observe the finite speed of propagation, $|\dot{x}_t| \lesssim 1$. However, for $|\xi| \sim N$, the angle is only approximately conserved up to times $N^{-\frac{1}{2}}$. This matches the time-scale on which the variable-coefficient oscillatory integral operator is well-approximated with the constant-coefficient Fourier extension operator.

\medskip

In the Hermite wave case the slow change as well in $x$ as in $\xi$ allows us to linearize the phase function for times \emph{independently} of the frequency scale. Moreover, we can fix the spatial scale to obtain a linearized phase function, which is of Klein-Gordon type (after normalizing the frequencies $|\xi| \sim 1$):
\begin{equation}
\label{eq:KleinGordonIntroduction}
\varphi(\xi) = \sqrt{|\xi|^2 + m^2}.
\end{equation}
$m$ depends on the modulus of $x_0$, which is the point in space around which we linearize, and the frequency scale $N$. 

\medskip

The analysis of the Klein-Gordon propagation for frequencies $|\xi| \sim N$ on the time-scale $|t| \lesssim 1$
\begin{equation*}
S_{m^2,N} f(x,t) = \int_{\R^d} e^{i(x \cdot \xi + t \varphi(\xi))} \hat{f}(\xi) \chi_N(\xi) d\xi
\end{equation*}
 splits into the wave regime, where $m^2 \lesssim N$, the elliptic regime, $N \lesssim m^2 \lesssim N^3$, and the stationary regime $m^2 \gtrsim N^3$. In the wave regime the propagation is indistinguishable from the wave propagation. When $m^2$ increases and becomes comparable to $1$, we are in the elliptic regime: The dispersive effects are strongest and the derivative loss for local smoothing becomes largest. Finally, for $m^2 \gg N$ the dispersive effects decrease again and for $m^2 \gtrsim N^3$, there is essentially no propagation anymore. 

Pointwise $L^p$-estimates, which are uniform in $N$ and $m^2$ read
\begin{equation*}
\| S_{m^2,N} f(1) \|_{L^p(\R^d)} \lesssim N^{d \big| \frac{1}{2} - \frac{1}{p} \big|} \| f \|_{L^p(\R^d)}.
\end{equation*}
Again, the maximal derivative loss stems from the elliptic regime $m \sim N$.
Integration in $L^p$ in time leads to smoothing.
Based on the Knapp examples from Section \ref{section:Knapp}, the local smoothing conjecture for the Klein--Gordon equations after frequency localization reads
\begin{equation*}
\| S_{m^2,N} f \|_{L^p_{t,x}([0,1] \times \R^d)} \lesssim N^{s} \| f \|_{L^p(\R^d)}
\end{equation*}
for
\begin{equation*}
s \geq \max( d \big| \frac{1}{2} - \frac{1}{p} \big| - \frac{1}{p},0 ).
\end{equation*}
Note that although strictly speaking the Knapp examples only apply to the linearized evolution of the Hermite wave equation, it is not too difficult to reverse the linearization by another Fourier series argument. This argument has many precursors, and we exemplarily refer to Beltran--Hickman--Sogge \cite[Lemma~2.6]{BeltranHickmanSogge2020} for details. We are led to the following:
\begin{conjecture}[Local~smoothing~conjecture~for~the~Hermite~wave~equation]
\label{conj:LocalSmoothing}
Let $d \geq 1$, $2 < p < \infty$, and $0<T<\infty$. Let $u$ be a solution to \eqref{eq:WaveHermiteEquation}. The estimate
\begin{equation*}
\| u \|_{L^p_{t,x}([-T,T] \times \R^d)} \lesssim_T \| u_0 \|_{L^p_s(\R^d)}
\end{equation*}
holds provided that
\begin{equation*}
s \geq \max \big( d \big| \frac{1}{2} - \frac{1}{p} \big| - \frac{1}{2}, 0 \big).
\end{equation*}
\end{conjecture}

By means of linearizing the parametrix from Section \ref{section:Preliminaries} and \ref{section:LinearizationReduction}, the proof will follow from local smoothing estimates for Klein-Gordon equations with dispersion relation given by \eqref{eq:KleinGordonIntroduction}. For the Schr\"odinger equation, Rogers \cite{Rogers2008} showed that local smoothing estimates
\begin{equation*}
\| e^{it \Delta} u_0 \|_{L^p_{t,x}([0,1] \times \R^d)} \lesssim \| u_0 \|_{L^p_s(\R^d)}
\end{equation*}
for $p \geq \frac{2(d+1)}{d}$ and $s > 2d \big( \frac{1}{2} - \frac{1}{p} \big) - \frac{2}{p}$ are equivalent to Fourier extension estimates
\begin{equation*}
\| \int_{\{ |\xi| \leq 1 \} } e^{i(x \cdot \xi + t |\xi|^2)} f(\xi) d\xi \|_{L^p_{t,x}(B_{d+1}(0,R))} \lesssim_\varepsilon R^\varepsilon \| f \|_{L^p}.
\end{equation*}
So, for the Schr\"odinger equation the restriction and local smoothing conjecture are indeed equivalent. Consequently, for Klein-Gordon equations we expect that proving the sharp local smoothing range is at least as hard as the restriction conjecture for more general elliptic surfaces.

Further remarks on local smoothing estimates for Schr\"odinger-like equations are in order: For fractional Schr\"odinger evolutions we refer to the work by Rogers--Seeger \cite{RogersSeeger2010}, in which moreover $L^p$-maximal estimates were proved. Lee--Rogers--Seeger \cite{LeeRogersSeeger2013} further elaborated on equivalence of Fourier restriction and local smoothing estimates for the Schr\"odinger equation.
On the other hand, it appears that assuming the phase space localization $\{|x| \sim |\xi| \sim N \}$, the above result can be improved using Bochner--Riesz square function estimates; see \cite{GanOhWu2021,Lee2018}.

\medskip

Next, we discuss our progress on Conjecture \ref{conj:LocalSmoothing}.
Once we have linearized the parametrix, we rely on square function and decoupling estimates to obtain local smoothing with derivative loss sharp up to endpoints. 
 In one dimension, we combine a square function estimate in $L^4$, which is adapted to $m^2 = |x_0|^2 / N^4$ with Kakeya maximal function estimates to prove the conjecture up to the endpoint:

\begin{theorem}[Local~smoothing~in~one~dimension]
\label{thm:LocalSmoothing1d}
Let $2 < p < \infty$, $s > \max(\frac{1}{2}-\frac{2}{p},0)$, and $u_0 \in L^p_{s}(\R)$. Let $u$ be a solution to \eqref{eq:WaveHermiteEquation}. Then the following estimate holds:
\begin{equation}
\label{eq:LocalSmoothing1dIntroduction}
\| u \|_{L_t^p([0,1],L_x^p(\R))} \lesssim_s \| u_0 \|_{L^p_{s}(\R)}.
\end{equation}
\end{theorem}

With the linearization described above in mind, we recover for essentially compactly supported initial data the local smoothing estimates for the wave propagation:
\begin{theorem}
\label{thm:CompactlySupportedInitialData}
Let $d \geq 2$ and $u_0 \in C^\infty_c(\R^d)$. Suppose that $u_0$ is supported in a ball of radius $R$ centered at the origin, and let $u$ be the solution to \eqref{eq:WaveHermiteEquation}. Then it holds:
\begin{equation*}
\| u \|_{L_t^p([0,1],L_x^p(\R^d))} \lesssim_{R} \| u_0 \|_{L^p_s(\R^d)}
\end{equation*}
provided that
\begin{equation*}
p \geq \frac{2(d+1)}{d-1} \text{ and } s > (d-1) \big( \frac{1}{2} - \frac{1}{p} \big) - \frac{1}{p}.
\end{equation*}
\end{theorem}
Under the assumption on the phase space localization, we find the linearization to resemble the Euclidean wave propagation. Then we can apply the Bourgain--Demeter decoupling estimates \eqref{eq:ConeDecouplingIntroduciton} for the cone.
We remark that in the case $d=2$ we can as well use the square function estimate due to Guth--Wang--Zhang \cite{GuthWangZhang2020} and maximal estimates to show local smoothing sharp up to endpoints.


\medskip

If we do not impose a phase space localization of the solution, the elliptic regime $\{ |x| \sim |\xi| \sim N \}$ will make a contribution.  In this case the derivative loss will be increased compared to Theorem \ref{thm:CompactlySupportedInitialData}:
\begin{theorem}
\label{thm:GeneralInitialData}
Let $d \geq 2$ and $u$ be a solution to \eqref{eq:WaveHermiteEquation}. Then the following estimate holds:
\begin{equation}
\label{eq:LocalSmoothingGeneralInitialData}
\| u \|_{L_{t,x}^p([0,1] \times \R^d)} \lesssim \| u_0 \|_{L^p_s(\R^d)}
\end{equation}
provided that
\begin{equation*}
p \geq \frac{2(d+2)}{d} \text{ and } s > d \big( \frac{1}{2} - \frac{1}{p} \big) - \frac{1}{p}.
\end{equation*}
\end{theorem}
For the proof we linearize and apply the decoupling theorem for elliptic surfaces. We show new decoupling estimates for elliptic surfaces, which are degenerate in the radial direction.

We supplement the results with examples, which show that the derivative loss cannot be improved. It seems reasonable to conjecture based on the local smoothing estimates for Schr\"odinger equations that \eqref{eq:LocalSmoothingGeneralInitialData} remains valid for
\begin{equation*}
p \geq \frac{2(d+1)}{d} \text{ and } s > d \big( \frac{1}{2} - \frac{1}{p} \big) - \frac{1}{p}.
\end{equation*}

\medskip

In this context, we comment on local smoothing estimates for the Hermite operator:
\begin{equation*}
\| e^{it \mathcal{H}} f \|_{L^p_{t,x}([-c,c] \times \R^d)} \lesssim \| f \|_{L^p_s(\R^d)}.
\end{equation*}
It turns out that for $c < \frac{\pi}{4}$, these are equivalent to the local smoothing estimates for the Schr\"odinger equation. This can be seen from the pseudo-conformal transform, also known as lens transform (cf. \cite{Thomann2009,BurqPoiretThomann2023,Tao2009}). 

The lens transform allows us to compare the solutions to the initial value problem with Hermite potential
\begin{equation*}
\left\{ \begin{array}{cl}
i \partial_t u + \mathcal{H} u &= 0, \quad (t,x) \in (-\frac{\pi}{4},\frac{\pi}{4}) \times \R^d, \\
u(0) &= u_0 
\end{array} \right.
\end{equation*}
to the solutions to the Schr\"odinger equation without potential:
\begin{equation*}
u(t,x) = \big( \frac{1}{\cos(2t)} \big)^{\frac{d}{2}} v \big( \frac{\tan(2t)}{2}, \frac{x}{\cos(2t)} \big) e^{-i \frac{|x|^2 \tan(2t)}{2}}.
\end{equation*}
Here $v$ denotes the solution to the Cauchy problem for the Schr\"odinger equation:
\begin{equation*}
\left\{ \begin{array}{cl}
i \partial_t v + \Delta v &= 0 , \quad (t,x) \in \R \times \R^d, \\
v(0) &= u_0
\end{array} \right.
\end{equation*}

Consequently, for $c < \frac{\pi}{4}$ it holds:
\begin{equation*}
\| u \|_{L^p_{t,x}([-c,c] \times \R^d)} \lesssim_c \| v \|_{L^p_{t,x}([-C,C] \times \R^d)}, \quad C= \frac{\tan(2c)}{2},
\end{equation*}
and we can use the smoothing estimates for the Schr\"odinger equation due to Rogers \cite{Rogers2008}. It is an interesting question to consider smoothing estimates for $c \geq \frac{\pi}{4}$. By the above, one could surmise that these correspond to global-in-time smoothing estimates for the Schr\"odinger equation. We remark that Bongioanni--Rogers \cite{BongioanniRogers2011} investigated the following smoothing estimates for $T < \infty$:
\begin{equation*}
\| e^{it \mathcal{H}} f \|_{L^p_x(\R^d, L_t^2([0,T]))} \lesssim \| \langle \mathcal{H} \rangle^{-s} f \|_{L^2(\R^d)}.
\end{equation*}

\medskip

We end the article by pointing out the relevance of Conjecture \ref{conj:LocalSmoothing} for the corresponding Bochner--Riesz problem.
In Section \ref{section:BochnerRiesz} we argue how local smoothing estimates have implications for the Bochner--Riesz means of the Hermite operator:
\begin{equation*}
\mathcal{B}_{\lambda}^{\alpha}(\mathcal{H}) f = \big( 1 - \frac{\mathcal{H}}{\lambda} \big)^{\alpha}_+ f, \quad x_+ = \max(x,0).
\end{equation*} 
We consider estimates for $\alpha > 0$, which are uniform in $\lambda \geq 1$.
In one dimension this recovers the classical result due to Askey--Wainger \cite{AskeyWainger1965} up to endpoints, which states that for $p=4$ it holds $L^p$-boundedness of $\mathcal{B}^{\alpha}_{\lambda}(\mathcal{H})$ for any $\alpha > 0$. In higher dimensions Conjecture \ref{conj:LocalSmoothing} implies that the critical $L^p$-space, $p>2$, for $\mathcal{B}^{\alpha}_{\lambda}(\mathcal{H})$ to be bounded for $\alpha > 0$ is given by $p_c=\frac{2(d+1)}{d}$, which is lower than the one from the Euclidean Bochner--Riesz conjecture given by $p_c = \frac{2d}{d-1}$. This matches the recent examples of Lee--Ryu \cite{LeeRyu2022}, who showed that summability of the global Bochner--Riesz means for the Hermite operator deviates from the Euclidean summability index. We shall see that even for the phase function obtained by Lee--Ryu \cite{LeeRyu2022} for local Hermite Bochner--Riesz means a certain curvature condition \cite{Bourgain1991,GuoWangZhang2022} fails in general. This indicates by recent results of Guo--Wang--Zhang \cite{GuoWangZhang2022} that the summability properties for local Hermite Bochner--Riesz means are inferior to the ones for Euclidean Bochner--Riesz means as well.

\bigskip

\textbf{Basic notations:}
\begin{itemize}
\item $\mathcal{S}(\R^d)$ denotes the Schwartz functions, i.e.,
\begin{equation*}
\mathcal{S}(\R^d) = \{ f \in C^\infty(\R^d;\C) : \forall \alpha, \beta \in \N_0^d: \, \sup_{x \in \R^d} |x^\alpha \partial_x^\beta f(x) | < \infty \}.
\end{equation*}
$\mathcal{S}'(\R^d)$ denotes the space of tempered distributions.
\item $\cdot : \R^d \times \R^d \to \R$ denotes the scalar product in $\R^d$, i.e., $x \cdot y = \sum_{i=1}^d x_i y_i$.
\item 
$L^p_s$ denotes the $L^p$-based Sobolev space of order $s$ given by
\begin{equation*}
L^p_s(\R^d) = \{ f \in L^p(\R^d) : \langle D \rangle^s f \in L^p(\R^d) \}
\end{equation*}
with $\langle D \rangle^s$ denoting the Bessel potential.
\item Capital numbers $N,M,\ldots \in 2^{\N_0}$ typically denote dyadic numbers.
\item $P_N$ denotes the smooth projection in Fourier space to frequencies of size $N$.
\item $\mathfrak{P}_N$ denotes the projection to Hermite eigenfunctions with eigenvalues comparable to $N^2$.
\item A $\beta$-sector $S_\beta$ is a subset of the unit annulus with aperture $\beta$ and length $1$. Suppose that it is centered around the line with direction $\nu \in \mathbb{S}^{d-1}$. Then it admits the parametrization
\begin{equation*}
S_\beta = \{ \xi \in \R^d : |\xi| \in [1/2,2], \; \big| \frac{\xi}{|\xi|} - \nu \big| \leq \beta \}.
\end{equation*}
\item An $(\alpha,\beta)$-sector $S_{\alpha,\beta}$ is a subset of the unit annulus with aperture $\beta$ and radial length $\alpha$. Suppose that it is angularly centered at $\nu \in \mathbb{S}^{d-1}$ and radially at $r \in [1/2,2]$. Then it admits the parametrization
\begin{equation*}
S_{\alpha,\beta} = \{ \xi \in \R^d : |\xi| \in [1/2,2] \cap [r-\frac{\alpha}{2},r+\frac{\alpha}{2}], \; \big| \frac{\xi}{|\xi|} - \nu \big| \leq \beta \}.
\end{equation*}
\item $P_\theta$ denotes smooth Fourier projection to the set $\theta \subseteq \R^d$ (e.g. a $\beta$-sector).
\end{itemize}

\medskip

\emph{Outline of the paper.} In Section \ref{section:Preliminaries} we recall basic facts about Hermite eigenfunctions and spectral localization. 
In Section \ref{section:LinearizationReduction} we analyze the eikonal equation. Using arguments related to the proof of the Cauchy-Kowalevskaya theorem we obtain an analytic expansion. We shall see that we can linearize the phase function after which we obtain a Klein-Gordon-type propagation \eqref{eq:KleinGordonIntroduction}. In Section \ref{section:Knapp} we work out Knapp examples, which yield necessary conditions for the local smoothing estimate.
In Section \ref{section:ProofLocalSmoothing1d} we show the local smoothing estimates in one dimension given by Theorem \ref{thm:LocalSmoothing1d}. In Section \ref{section:ProofLocalSmoothingHigher} we show decoupling estimates for elliptic surfaces which are degenerate in the radial direction. Based on suitable decoupling estimates, we show Theorem \ref{thm:CompactlySupportedInitialData}, which recovers the local smoothing estimates for the wave propagation in case of essentially compactly supported initial data. Secondly we show Theorem \ref{thm:GeneralInitialData}, which covers the elliptic regime as well. Finally, in Section \ref{section:BochnerRiesz} we connect Conjecture \ref{conj:LocalSmoothing} to global Hermite Bochner--Riesz means and remark on local Hermite Bochner--Riesz means.

\section{Preliminaries}
\label{section:Preliminaries}
\subsection{Hermite eigenfunctions and spectral localization}

In the following we recall basic facts about the Hermite eigenfunctions (cf. \cite{AbramowitzStegun1964}). Note that $\mathcal{H}_d = - \Delta_d + |x|^2 = \sum_{j=1}^d (- \partial_j^2 + x_j^2)$ decomposes into $d$ commuting one-dimensional Harmonic oscillators. The one-dimensional oscillator $\mathcal{H}_1 = - \partial_x^2 + |x|^2$ has discrete spectrum 
\begin{equation*}
\sigma(\mathcal{H}_1) = \{ 2k + 1 \, : \, k \in \N_0 \}
\end{equation*}
with $L^2$-normalized eigenfunctions given by
\begin{equation*}
h_n(x) = (-1)^n (2^n n! \sqrt{\pi})^{-\frac{1}{2}} e^{\frac{x^2}{2}} \frac{d^n}{dx^n} e^{-x^2}.
\end{equation*}
Consequently, we obtain eigenfunctions to $\mathcal{H}_d$ by tensorization
\begin{equation*}
\mathfrak{h}_{\underline{n}}(x_1,\ldots,x_d) = h_{n_1}(x_1) \ldots h_{n_d}(x_d), \quad \underline{n} = (n_1,\ldots,n_d) \in \N_0^d.
\end{equation*}
These satisfy $\mathcal{H}_d \mathfrak{h}_{\underline{n}} = (2(n_1+\ldots+n_d) +d) \mathfrak{h}_{\underline{n}}$ such that we obtain
\begin{equation*}
\sigma(\mathcal{H}_d) = \{ 2k + d \, : \, k \in \N_0 \}.
\end{equation*}
But note that the eigenspaces are no longer one-dimensional except from the lowest eigenvalue $d$. For $s \geq 0$, powers $\mathcal{H}^s$ and the evolution $e^{it \mathcal{H}^s}:L^2(\R^d) \to L^2(\R^d)$ are defined by spectral calculus.

\medskip

For $k \in \sigma(\mathcal{H}_d)$ we define the orthonormal projection $\mathfrak{p}_k : L^2(\R^d) \to L^2(\R^d)$ to the eigenspace with eigenvalue $k$. For $N \in 2^{\N_0}$ we define
\begin{equation*}
\mathfrak{P}_N = \sum_{\substack{k \in \sigma(\mathcal{H}_d), \\
\frac{N^2}{4} \leq k \leq 4 N^2}} \mathfrak{p}_k, \qquad \mathfrak{P}_{\leq N} = \sum_{\substack{k \in \sigma(\mathcal{H}_d), \\ k \leq 4 N^2}} \mathfrak{p}_k.
\end{equation*}
Here $N$ refers to the spectral localization of $\sqrt{\mathcal{H}}$. 

\medskip

We define the standard frequency projection $P_N : L^2(\R^d) \to L^2(\R^d)$, $N \in 2^{\N_0}$, as Fourier multipliers. Let $\chi_0 \in C^\infty_c(B_d(0,2))$ be a radially decreasing function with $\chi_0(\xi) = 1$ for $|\xi| \leq 1$, and $\chi(\xi) = \chi_0(\xi/2) - \chi_0(\xi)$. We define for $N \in 2^{\N_0}$
\begin{equation*}
(P_N f) \widehat (\xi) = \chi_N(\xi) \hat{f}(\xi), \quad \chi_N(\xi) = \chi(\xi/N), \quad (P_0 f) \widehat(\xi) = \chi_0(\xi) \hat{f}(\xi).
\end{equation*}

\subsection{Estimates for low Hermite frequencies}
\label{subsection:LowHermiteFrequencies}
To avoid the singularity of the symbol $p(x,\xi)= \sqrt{\xi^2 + |x|^2}$ at the origin $x=\xi=0$, we shall only consider initial data with $\mathfrak{P}_{\leq C} f = 0$. The contribution of this part is clearly bounded in $L^p$:
\begin{equation*}
\begin{split}
\| \mathfrak{P}_{\leq C} e^{it \sqrt{\mathcal{H}}} f \|_{L^p} &\leq \sum_{k \leq C} \| \mathfrak{p}_k e^{it \sqrt{\mathcal{H}}} f \|_{L^p} \\
&\leq \sum_{k \leq C} \| \mathfrak{p}_k e^{it k} f \|_{L^p} \leq \sum_{k \leq C} \| \mathfrak{p}_k f \|_{L^p} \lesssim \|f \|_{L^p}.
\end{split}
\end{equation*}
In the ultimate estimate we used $L^p$-boundedness of $\mathfrak{p}_k$ for $k \leq C$, which is immediate from $h_m \in \mathcal{S}(\R)$, and that there are only finitely many projections $k \leq C$\footnote{The number depends on the dimension though.}. Consequently, we suppose in the remainder of the analysis that $f = \mathfrak{P}_{\geq C} f$. Let $\tilde{P}_{\lesssim 1} = \tilde{p}(x,D)$ denote a pseudo-differential operator which localizes in phase space to $\{ |x| + | \xi | \lesssim 1 \}$. We can estimate the contribution of the small frequencies as follows:
\begin{equation}
\label{eq:LowFrequencyDecomposition}
\begin{split}
\| \mathfrak{P}_{\geq C} e^{it \sqrt{\mathcal{H}}} \tilde{P}_{\lesssim 1} \mathfrak{P}_{\geq C} f \|_{L^p} &\leq \sum_{N \gg 1} \| \mathfrak{P}_N e^{it \sqrt{\mathcal{H}}} \tilde{P}_{\lesssim 1} \mathfrak{P}_{\geq C} f \|_{L^p} \\
&\lesssim \sum_{N,M \gg 1} N^m \| \mathfrak{P}_{N} \tilde{P}_{\lesssim 1} \mathfrak{P}_{M} f \|_{L^p}.
\end{split} 
\end{equation}
Since by \cite[Proposition~A.5]{BurqPoiretThomann2023} we have
\begin{equation*}
\| \mathfrak{P}_N \tilde{P}_{\lesssim 1} \|_{L^p \to L^p} + \| \tilde{P}_{\lesssim 1} \mathfrak{P}_N \|_{L^p \to L^p} \lesssim_m N^{-m},
\end{equation*}
and clearly $\| \mathfrak{P}_N \|_{L^p \to L^p} \lesssim N^{c(d)}$, we can bound
\begin{equation*}
\| \mathfrak{P}_{N} \tilde{P}_{\lesssim 1} \mathfrak{P}_{M} f \|_{L^p} \lesssim \max(N,M)^{-m} \| f \|_{L^p}.
\end{equation*}
Then the summation over $N$ and $M$ in \eqref{eq:LowFrequencyDecomposition} can be carried out easily.

\subsection{Existence of solutions to the eikonal equation}
\label{subsection:ExistenceEikonal}

By the above we shall only consider initial data which have Fourier support away from the origin. We summarize the properties of the approximate solution obtained from the Lax parametrix \cite{Lax1957} and composition of Fourier integral operators:
\begin{proposition}
Let $N \gg 1$ and $\text{supp}(\hat{f}) \subseteq A_N = B(0,2N) \backslash B(0,N/2)$. Then the solution to the Cauchy problem
\begin{equation}
\label{eq:CauchyProblemProposition}
\left\{ \begin{array}{cl}
\partial_t^2 u &= \Delta u - |x|^2 u, \quad (x,t) \in \R^d \times [0,1], \\
u(0) &= f, \quad \dot{u}(0) = 0
\end{array} \right.
\end{equation}
is given by
\begin{equation*}
u(t,x) = \sum_{i=1}^2 \int_{\R^d} e^{i \phi_i(x,t;\xi)} a_i(x,t;\xi) \hat{f}(\xi) d\xi + R_t f
\end{equation*}
where $\phi_i$ solve the eikonal equation:
\begin{equation}
\label{eq:EikonalEquationProposition}
\left\{ \begin{array}{cl}
\partial_t \phi_i(x,t;\xi) &= (-1)^{i+1} p(x,\nabla_x \phi_i), \\
\phi_i(x,0;\xi) &= x \cdot \xi
\end{array} \right.
\end{equation}
with $p(x,\xi)= \sqrt{|\xi|^2 + |x|^2}$. Moreover, $\| R_t \|_{L^p \to L^p} \lesssim 1$ uniformly in $t \in [0,1]$ and $a_i \in S^0$ uniformly in $N$ and $t$.
\end{proposition}
This is an instance of \cite[Theorem~5.5]{Rauch2012}. We shall analyze the phase functions in detail below and remark on the amplitude function in Subsection \ref{subsection:Amplitude}.

\smallskip

Due to finite speed of propagation exhibited by the Hamilton flow, as well the spatial support as the frequency support is essentially preserved by the evolution of \eqref{eq:CauchyProblemProposition}. Secondly, by boundedness of $R_t$ in $L^p$ and symmetry of $\phi_1$ and $\phi_2$ established through time-reversal,
it suffices to analyze the Fourier integral operator given by
\begin{equation*}
u(t,x) = \int_{\R^d} e^{i \phi(x,t;\xi)} a(x,t;\xi) \hat{u}_0(\xi) d\xi
\end{equation*}
with $\phi$ a solution to \eqref{eq:EikonalEquationProposition} for $i=1$.

In the following we invoke the Cauchy-Kowalevskaya theorem to obtain solutions to the eikonal equation with $p(x,\xi) = \sqrt{|x|^2+ |\xi|^2}$:
\begin{equation}
\label{eq:EikonalExistence}
\left\{ \begin{array}{cl}
\partial_t \varphi(x,t;\xi) &= p(x,\nabla_x \varphi), \\
\varphi(x,0;\xi) &= x \cdot \xi.
\end{array} \right.
\end{equation}
We show the following:
\begin{proposition}[Existence~of~solutions~to~the~eikonal~equation]
\label{prop:ExistenceEikonal}
Let $(x,\xi) \in \R^{2d} \backslash 0$. There are analytic solutions to \eqref{eq:EikonalExistence} for $(x',t',\xi') \in \R^d \times \R \times \R^d$ in a ball $B((x,0,\xi),R)$ centered at $(x,0,\xi)$ of radius $R$ comparable to $|(x,\xi)|$.
\end{proposition}
\begin{proof}
This will be a consequence of the analytic domination argument in the proof of the Cauchy-Kowalevskaya theorem. We follow the presentation of Evans in \cite[Section~4.6]{Evans2010}.
Denote the phase-space variables with $z = (x,\xi) \in \R^{2d}$ and rewrite the eikonal equation as a linear system:
\begin{equation*}
\underline{\mathbf{u}} = 
\begin{pmatrix}
\varphi(x,t;\xi) \\ \partial_{x_1} \varphi(x,t;\xi) \\ \vdots \\ \partial_{x_n} \varphi(x,t;\xi)
\end{pmatrix}
= \begin{pmatrix}
\mathbf{u}_0 \\ \underline{\mathbf{u}}'
\end{pmatrix}
.
\end{equation*}
Note that $\underline{\mathbf{u}}$ has $m=n+1$ components.
The eikonal equation is equivalent to the linear system:
\begin{equation*}
\partial_t \underline{\mathbf{u}} = 
\begin{pmatrix}
p(x,\underline{\mathbf{u}}') \\
\nabla_{\xi} p(x,\underline{\mathbf{u}}') \cdot \partial_{x_1} \underline{\mathbf{u}}' + \partial_{x_1} p(x,\underline{\mathbf{u}}') \\
\vdots \\
\nabla_{\xi} p(x,\underline{\mathbf{u}}') \cdot \partial_{x_n} \underline{\mathbf{u}}' + \partial_{x_d} p(x,\underline{\mathbf{u}}')
\end{pmatrix}
.
\end{equation*}
This is rewritten concisely as
\begin{equation*}
\partial_t \underline{\mathbf{u}} = \sum_{j=1}^n \mathbf{B}_j(x,\underline{\mathbf{u}}') \partial_{x_j} \underline{\mathbf{u}}' + \underline{\mathbf{c}}(x,\underline{\mathbf{u}}').
\end{equation*}
The components of $\underline{\mathbf{c}}$ and $\mathbf{B}_j$ are given by $p$ and first order derivatives of $p$. After harmless linear translations, we can reduce to finding solutions to the homogeneous system centered at the origin.  The homogeneous system reads
\begin{equation*}
\left\{ \begin{array}{cl}
\partial_t \underline{\mathbf{u}} &= \sum_{j=1}^d \mathbf{B}_j(x,\underline{\mathbf{u}}') \partial_{x_j} \underline{\mathbf{u}} + \underline{\mathbf{c}}(x,\underline{\mathbf{u}}) \text{ for } |(z,t)| < r, \\
\underline{\mathbf{u}} &= 0 \text{ for } |z| < r.
\end{array} \right.
\end{equation*}
The radius of convergence of $\mathbf{B}_j$ and $\mathbf{c}$ depends on the centre $z_0 = (x_0,\xi_0)$ of the ball on which we solve the equation. Indeed, the radius of convergence of the real-valued scalar function $f(a) = \sqrt{r^2 + a^2}$ is comparable to $r$. We require an analytic majorization of $f(z) = \sqrt{(z-z_0)^2}$, $z \in \R^{2d}$. Note that this function has analyticity radius $r \sim |z_0|$. Let $D^\alpha f(0) = f_\alpha$. Moreover, we have the estimate
\begin{equation*}
|f_\alpha y^\alpha | \lesssim |z_0| =: C \text{ for } |y| \lesssim |z_0|.
\end{equation*}
This follows from writing
\begin{equation*}
f(z) = |z_0| \sqrt{\big( \frac{z}{|z_0|} - \frac{z_0}{|z_0|} \big)^2}.
\end{equation*}
Like in \cite{Evans2010} we find the analytic majorization:
\begin{equation*}
\left\{ \begin{array}{cl}
\partial_t \underline{\mathbf{u}}^* &= \frac{C r}{r- (x_1+\ldots+x_n+\xi_1+\ldots+\xi_n) - (u^{0*} + \ldots + u^{m*})} \big( \sum_{j,l} \mathbf{u}^{l *}_{j} + 1 \big), \\
\underline{\mathbf{u}}^* &= 0
\end{array} \right.
\end{equation*}
with explicit solution given by
\begin{equation*}
\underline{\mathbf{u}}^* = v^*(x,t;\xi) (1,\ldots,1)^t
\end{equation*}
for
\begin{equation*}
\begin{split}
v^*(x,t;\xi) &= \frac{1}{m(2n+1)} \big( r - (x_1+\ldots+x_n+\xi_1+\ldots+\xi_n+t) \\
&\quad - [(r-(x_1+\ldots+x_n+\xi_1+\ldots+\xi_n+t))^2 - 2 m(2n+1) C r t]^{\frac{1}{2}} \big).
\end{split}
\end{equation*}
This is analytic for $|(x,\xi)| \lesssim |(x_0,\xi_0)|$ and $|t| \lesssim |(x_0,\xi_0)|$. Consequently, for $|(x,\xi)| \gtrsim 1$, we obtain an analytic solution to
\begin{equation*}
\left\{ \begin{array}{cl}
\partial_t \varphi(x,t;\xi) &= p(x,\nabla_x \varphi), \quad (x,\xi) \in \R^{2d}, \\
\varphi(x,0;\xi) &= x \cdot \xi
\end{array} \right.
\end{equation*}
for $t \in [-c,c]$.
\end{proof}

\subsection{Asymptotics}
\label{subsection:Asymptotics}
In the following we use the analyticity to obtain asymptotics of the solution:
\begin{equation*}
\varphi(x,t;\xi) = x \cdot \xi + t \sqrt{|\xi|^2 + |x|^2} + \mathcal{E}(x,t;\xi).
\end{equation*}
Let $z=(x,\xi)$ for brevity.
Partial derivatives are denoted by $\partial_x^\alpha$,  $\partial_\xi^\beta$, $\alpha,\beta \in \N_0^d$, or $\partial_z^\gamma$, $\gamma \in \N_0^{2d}$. The Hessian is denoted by $\partial^2_{xx} \varphi = (\partial^2_{x_i x_j} \varphi)_{1 \leq i,j \leq d}$.
To analyze the remainder term, which is formally given by
\begin{equation*}
\mathcal{E}(x,t;\xi) =  \sum_{k \geq 2} \frac{t^k}{k!} \partial_t^k \varphi(x,0;\xi),
\end{equation*}
we iterate the equation and obtain:
\begin{equation}
\label{eq:SecondDerivativeEikonal}
\begin{split}
\partial_t^2 \varphi(x,t;\xi) &= \partial_\xi p(x,\nabla_x \varphi) \nabla_x \partial_t \varphi = \partial_\xi p(x,\nabla_x \varphi) \cdot \nabla_x (p(x,\nabla_x \varphi)) \\
&= \partial_{\xi} p(x,\nabla_x \varphi) \cdot [ \nabla_x p(x,\nabla_x \varphi) + \partial^2_{xx} \varphi \; \nabla_\xi p(x,\nabla_x \varphi) ].
\end{split}
\end{equation}
Note that $\partial_{xx}^2 \varphi (x,0;\xi) = 0$ and hence,
\begin{equation*}
\partial_t^2 \varphi(x,0;\xi)= \frac{x \cdot \xi}{|x|^2 + |\xi|^2}.
\end{equation*}

By induction, we can show the following:
\begin{lemma}
\label{lem:ExpansionEikonalTimeDerivatives}
For $k \geq 2$, we have the following representation of $\varphi$ solving \eqref{eq:EikonalEquationProposition}:
\begin{equation}
\label{eq:ExpansionEikonalSolution}
\partial_t^k \varphi(x,t;\xi) = \sum_{\substack{ N_1 =k, \\ 0 \leq N_2 \leq k-1}} \sum_{ (*)} c_{\underline{\alpha},\underline{\beta}} \, \prod_{i=1}^{N_1} \partial_z^{\alpha_i} p(x,\nabla_x \varphi) \, \prod_{j=1}^{N_2} \partial_x^{\beta_j} \varphi(x).
\end{equation}

The second sum ranges over
\begin{equation*}
\begin{split}
&(*) : (\alpha_i)_{i=1}^{N_1} \subseteq \N_0^{2d} \text{ and } (\beta_j)_{j=1}^{N_2} \subseteq \N_0^d: \\
&\quad \sum_{i=1}^{N_1} |\alpha_i| + \sum_{j=1}^{N_2} |\beta_j| = 2(k-1) + 2N_2, \quad \forall j: |\beta_j| \geq 2.
\end{split}
\end{equation*}

\end{lemma}
We remark that the terms with $N_2 = 0$ are most important, as these correspond to terms without factors of $\partial_x^\gamma \varphi$. For this reason these are the only contributing terms at $t=0$ as $|\beta_j| \geq 2$. 

As an auxiliary result we require the following:
\begin{lemma}
\label{lem:AuxiliaryDerivatives}
Let $\gamma \in \N_0^d$, $n_0 = |\gamma| \geq 1$. We have the following expansion for $\varphi$ a solution to \eqref{eq:EikonalEquationProposition}:
\begin{equation}
\label{eq:IdentityPhiDerivatives}
\partial_t \partial_x^{\gamma} \varphi(x,t;\xi) = \sum_{(**)} d_{\delta,\underline{\beta}'} \partial_z^{\delta} p(x,\nabla_x \varphi) \prod_{j=1}^{N_2'} \partial_x^{\beta'_j} \varphi.
\end{equation}
The second sum ranges over
\begin{equation*}
\begin{split}
&\quad (**) : \, \delta \in \N_0^{2d}:1 \leq |\delta| \leq n_0 \\ 
&\text{ and } (\beta'_j)_{j=1}^{N_2'} \subseteq \N_0^d: n_0 + 2N_2' = |\delta| + \sum_{j=1}^{N'_2} |\beta'_j|, \quad \forall j: 2 \leq |\beta'_j|.
\end{split}
\end{equation*}
\end{lemma}
\begin{proof}
This follows from induction on $n_0$ and analyticity of $\varphi$.
For $n_0=1$ observe
\begin{equation*}
\partial_t \nabla_x \varphi(x,t;\xi) = \nabla_x (p(x,\nabla_x \varphi)) = \nabla_x p(x,\nabla_x \varphi)+ \partial^2_{xx} \varphi(x,t;\xi) \nabla_{\xi} p(x,\nabla_x \varphi).
\end{equation*}

For the induction step we take a derivative in $x$ of the representation in \eqref{eq:IdentityPhiDerivatives}: When the derivative hits $\partial_z^\delta p(x,\nabla_x \varphi)$ we obtain
\begin{equation*}
\partial_x \partial_z^\delta p(x,\nabla_x \varphi) = (\partial_z^\delta \partial_x p)(x,\nabla_x \varphi) + \partial^2_{xx} \varphi \; (\partial_{\xi} \partial_z^\delta p)(x,\nabla_x \varphi) .
\end{equation*}
Clearly, the first term is accommodated by raising $|\delta| \to |\delta| + 1$. For the second term we have to raise $N_2' \to N_2' + 1$ as well and note that this term is covered in the representation $(**)$ for $|\gamma| = n_0 + 1$.

\smallskip

When the $x$ derivative hits the product $\prod_{j=1}^{N_2'} \partial_x^{\beta'_j} \varphi$, then the number $N_2'$ does not change but we need to raise one $\beta_j'$ by one. This is covered as well.
\end{proof}
Now we are ready to prove Lemma \ref{lem:ExpansionEikonalTimeDerivatives}.
\begin{proof}[Proof~of~Lemma~\ref{lem:ExpansionEikonalTimeDerivatives}]
For $k =2$ the above representation is valid as follows from \eqref{eq:SecondDerivativeEikonal}. Suppose it holds for $k \geq 2$. We check the representation for $k+1$ by taking the time derivative of \eqref{eq:ExpansionEikonalSolution}:
When the time derivative falls on the derivative of $p$, it follows
\begin{equation*}
\partial_t \partial_z^{\alpha_i} p(x,\nabla_x \varphi )= (\partial_\xi \partial_z^{\alpha_i} p)(x,\nabla_x \varphi) (\nabla_x p(x,\nabla_x \varphi) + \partial^2_{x} \varphi \nabla_\xi p(x,\nabla_x \varphi)).
\end{equation*}
To accommodate the first term in the expansion for $k+1$, we need to increase $N_1 \to N_1+1$ and $\sum_{i=1}^{N_1} |\alpha_i| \to \sum_{i=1}^{N_1} |\alpha_i| + 2$. To take the second term into account, increase $N_1 \to N_1+1$, $\sum_{i=1}^{N_1} |\alpha_i| \to \sum_{i=1}^{N_1} |\alpha_i| + 2$ and $N_2 \to N_2+1$, $\sum_{j=1}^{N_2} |\beta_j| \to \sum_{j=1}^{N_2} |\beta_j| + 2$.

We turn to the case when the time derivative hits $\partial_x^{\beta_j} \varphi$ with $|\beta_j| \leq 2k$. To ease notation, let $j=1$. In this case we can use Lemma \ref{lem:AuxiliaryDerivatives}:
\begin{equation*}
\partial_t \partial_x^{\beta_1} \varphi(x,t;\xi) = \sum_{(**)} d_{\underline{\delta},\underline{\beta}'} \partial_z^{\delta} p(x,\nabla_x \varphi) \prod_{j=1}^{N'_2} \partial_x^{\beta'_j} \varphi.
\end{equation*}
This means we have to increase $N_1 \to N_1 +1$ and $N_2 \to N_2 + N_2'$ and $_{i=1}^{N_1} |\alpha_i| \to _{i=1}^{N_1} |\alpha_i| + |\delta|$, $\sum_{j=1}^{N_2} |\beta_j| \to \sum_{j=2}^{N_2} |\beta_j| + \sum_{j'=1}^{N_2'} |\beta_{j'}|$.
By the induction assumption we have
\begin{equation*}
\sum_{i=1}^{N_1} |\alpha_i| + \sum_{j=1}^{N_2} | \beta_j| = 2(k-1) + 2N_2, \quad \forall j: |\beta_j| \geq 2.
\end{equation*}

It follows from Lemma \ref{lem:AuxiliaryDerivatives} that
\begin{equation*}
|\delta| + \sum_{j=1}^{N_2'} |\beta_j'| = |\beta_1| + 2 N_2', \quad \forall j: |\beta_{j}'| \geq 2.
\end{equation*}
Consequently,
\begin{equation*}
\begin{split}
\sum_{i=1}^{N_1} |\alpha_i| + |\delta| + \sum_{j=2}^{N_2} |\beta_j| + \sum_{j=1}^{N_2'} |\beta_j'| &= \sum_{i=1}^{N_1} |\alpha_i| +  \sum_{j=1}^{N_2} |\beta_j| + 2 N_2' \\ &= 2 k + 2(N_2-1) + 2N_2'
\end{split}
\end{equation*}
and for all $j$ in the above display, we have $|\beta_j|, |\beta_j'| \geq 2$. The proof is complete.

\end{proof}

Consequently, evaluating the expression at zero, we obtain
\begin{equation*}
\partial_t^k \varphi(x,0;\xi) = \sum_{1 \leq N_1 \leq k} \sum_{ \sum_{i=1}^{N_1} |\alpha_i| = 2(k-1) } c_{\underline{\alpha},\underline{0}} \prod_{i=1}^{N_1} \partial_z^{\alpha_i} p(x,\xi).
\end{equation*}
The terms with $N_2 \geq 1$ are clearly not contributing because $\partial_{x}^{\beta_j} \varphi(x,t;\xi) = 0$ for $|\beta_j| \geq 2$. For later purposes it will be very useful to have convergence of the modified series:
\begin{equation}
\label{eq:ModifiedQuantity}
B =  \sum_{k=2}^\infty \frac{t^k}{k!} b_k, \quad b_k = \sum_{1 \leq N_1 \leq k} \sum_{ \sum_{i=1}^{N_1} |\alpha_i| = 2(k-1) } c_{\underline{\alpha},\underline{0}} \prod_{i=1}^{N_1} |\partial_z^{\alpha_i} p(x,\xi)|.
\end{equation}

To this end, we solve the eikonal equation with different symbol $\tilde{p}(x,\xi)$:
\begin{equation}
\label{eq:ModifiedEikonalEquation}
\left\{ \begin{array}{cl}
\partial_t \tilde{\varphi}(x,t;\xi) &= \tilde{p}(x,\nabla_x \tilde{\varphi}), \quad (x,t,\xi) \in \R^d \times \R \times \R^d, \\
\tilde{\varphi}(x,0;\xi) &= x \cdot \xi .
\end{array} \right.
\end{equation}
We choose $\tilde{p}$ to be a majorization of $p$ around $(x_0,\xi_0)$ with analyticity radius comparable to $r = |(x_0,\xi_0)|$. Then we use the above arguments, rewriting the eikonal equation as a linear system to obtain an analytic solution around $(x_0,0,\xi_0)$ with analyticity radius comparable to $r$. This solution can be expanded as
\begin{equation*}
\tilde{\varphi}(x,t;\xi) = \langle x, \xi \rangle + t \tilde{p}(x,\xi) + \sum_{k \geq 2} \frac{t^k}{k!} \sum_{1 \leq N_1 \leq k} \sum_{ \sum_{i=1}^{N_1} |\alpha_i| = 2(k-1) } c_{\underline{\alpha},\underline{0}} \prod_{i=1}^{N_1} \partial_z^{\alpha_i} \tilde{p}(x,\xi).
\end{equation*}
Consequently, the solution $\tilde{\varphi}$ to \eqref{eq:ModifiedEikonalEquation} yields the convergence of \eqref{eq:ModifiedQuantity}.

\bigskip
\begin{remark}
\label{rem:LongTimeExistence}
Finally, we remark that the solution to the eikonal equation exists for $t \in [-1,1]$ choosing the frequency large enough. This follows from rewriting by homogeneity
\begin{equation*}
\begin{split}
\varphi_N(x',t';\xi') &= x' \cdot \xi' + t' \sqrt{|\xi'|^2 + |x'|^2} \\
&\quad + \sum_{k \geq 2} \frac{(t')^k}{N^k k!} \sum_{\sum_{i=1}^k |\alpha_i| = 2(k-1)} c_{\underline{\alpha},\underline{0}} N^{2-k} \prod_{i=1}^k \partial_z^{\alpha_i} p(N^{-2} x', \xi').
\end{split}
\end{equation*}
By the analysis of the previous section the expression
\begin{equation*}
\sum_{k \geq 2} \frac{(t'')^k}{k!} \sum_{\sum_{i=1}^d |\alpha_i| = 2(k-1)} c_{\underline{\alpha},\underline{0}} \prod_{i=1}^k \partial_z^{\alpha_i} p(N^{-2} x', \xi')
\end{equation*}
converges for $|t''| \lesssim 1$. This means $\varphi_N(x',t';\xi')$ exists for times $|t'| \lesssim N^2$.
 This yields existence of $\varphi_N(x',t';\xi')$ for $|t'| \leq N$ choosing $N$ large enough.
\end{remark}

\subsection{Amplitude}
\label{subsection:Amplitude}
The amplitude is obtained by successively solving transport equations. The precise form of the amplitudes is not important in the following, and we refer to \cite[Section~5.5]{Rauch2012} for the construction. We shall need that $a_i \in S^0$ uniformly in $|(x,\xi)| \gtrsim 1$ and existence of $a_i$ for times $|t| \gtrsim 1$. This follows again from the scaling considerations.

\section{Linearization and Reduction to Klein-Gordon propagation}
\label{section:LinearizationReduction}
\subsection{Linearization}
\label{subsection:Linearization}
We consider $|t| \lesssim 1$, and $|x-x_0| \lesssim 1$ by finite speed of propagation.
We suppose that $\{ |\xi| \sim N \}$. To facilitate the linearization, we normalize the frequencies to the unit annulus and correspondingly, rescale $x$ and $t$ by $N$. Let
\begin{equation*}
\varphi(x,t;\xi) = x \cdot \xi + t \sqrt{|\xi|^2 + |x|^2} + \sum_{k \geq 2} \frac{t^k}{k!} \partial_t^k \varphi(x,0;\xi).
\end{equation*}
We define
\begin{equation*}
\begin{split}
&\quad \varphi_N(x',t';\xi') \\
&= \varphi(N^{-1} x', N^{-1} t', N \xi') \\
&= x' \cdot \xi' + N^{-1} t' \sqrt{N^2 |\xi'|^2 + N^{-2} |x'|^2} + \sum_{k \geq 2} \frac{(t')^k}{N^k k!} (\partial_t^k \varphi) (N^{-1} x',0; N \xi') \\
&= x' \cdot \xi' + t' \sqrt{|\xi'|^2 + N^{-4} |x'|^2} + \sum_{k \geq 2} \frac{(t')^k}{N^k k!} (\partial_t^k \varphi)(N^{-1} x',0;N \xi').
\end{split}
\end{equation*}
After rescaling, we have $\{|\xi| \sim 1\}$ and $|t| \lesssim N$. Invoking finite speed of propagation, we can actually suppose that $\{|x-x_0| \lesssim N \}$. 

\medskip

To find size and regularity estimates for
\begin{equation*}
\partial_t^k \varphi(x,0;\xi) = \sum_{\sum_{i=1}^{k} |\alpha_i| = 2(k-1)} c_{\underline{\alpha},\underline{0}} \prod_{i=1}^{k} \partial_z^{\alpha_i} p(N^{-1} x', N \xi'),
\end{equation*}
we take advantage of homogeneity of $p$. Write $p(z) = |z|$, $z \in \R^{2d}$. Note that $\partial_z^\alpha p$ is homogeneous of order $1-|\alpha|$, and we have the simple estimate
\begin{equation}
\label{eq:DerivativesModulus}
|\partial_z^\alpha p(z) | \leq 2 |\alpha| \text{ for } |z| \gtrsim 1 \text{ and } |\alpha| \geq 1. 
\end{equation}

By homogeneity we find
\begin{equation*}
\partial_z^\alpha p(N^{-1} x', N \xi') = N^{1-|\alpha|} \partial_z^\alpha p(N^{-2} x', \xi'),
\end{equation*}
and we obtain from taking the product by the sum constraint on $\alpha_i$:
\begin{equation*}
\begin{split}
\prod_{i=1}^{k} \partial_z^{\alpha_i} p(N^{-1} x', N \xi') &= N^{k-\sum_{i=1}^k |\alpha_i|} \prod_{i=1}^{k} \partial_z^{\alpha_i} p(N^{-2} x', \xi') \\
&= N^{2-k} \prod_{i=1}^{k} \partial_z^{\alpha_i} p(N^{-2} x', \xi').
\end{split}
\end{equation*}

Next, we shall check on which domain we can linearize $\varphi_N(x,t;\xi)$. By linearization we mean an expansion of the oscillatory integral into Fourier extension operators. Note that the maximal domain we need to consider is $|x-x_0| \lesssim N$, $|t| \lesssim N$ by finite speed of propagation. We decompose:
\begin{equation*}
\varphi_N(x,t;\xi) = x \cdot \xi + t \sqrt{|\xi|^2 + |x_0|^2 / N^4} + \mathcal{E}_N(x,t;\xi)
\end{equation*}
with
\begin{equation*}
\begin{split}
\mathcal{E}_N(x,t;\xi) &= t \big( \sqrt{|\xi|^2 + |x|^2/N^4} - \sqrt{|\xi|^2 + |x_0|^2 / N^4} \big) \\
&\quad + \sum_{k \geq 2} \frac{N^{2(1-k)} t^k}{k!} (\partial_t^k \varphi)(N^{-2} x, 0;\xi) \\
&=: \mathcal{E}_{N1}(x,t;\xi) + \mathcal{E}_{N2}(x,t;\xi).
\end{split}
\end{equation*}

For $\mathcal{E}_N(x,t;\xi)$ we obtain boundedness of the derivatives:

\begin{lemma}
\label{lem:EstimatesRemainderTerm}
Let $N \gg 1$. With notations from above, the following estimate holds:
\begin{equation}
\label{eq:RegularityHigherOrder}
\sup_{\substack{x \in \R^d, \, |\xi| \sim 1, \\ |t| \lesssim N, \, |x-x_0| \lesssim N}} |\partial_\xi^\alpha \mathcal{E}_N(x,t;\xi) | \lesssim_{\alpha,d} 1 \text{ for } 0 \leq |\alpha| \leq d+1.
\end{equation}
\end{lemma}
\begin{proof}
The linear term in $t$ is estimated by noticing
\begin{equation*}
\begin{split}
\sqrt{|\xi|^2 + |x|^2 / N^4} - \sqrt{|\xi|^2 + |x_0|^2 / N^4} &= \frac{(|x|^2 - |x_0|^2)/N^4}{\sqrt{|\xi|^2 + |x|^2 / N^4} + \sqrt{|\xi|^2 + |x_0|^2/N^4}} \\
&= \frac{(|x| - |x_0|)(|x|+ |x_0|)/N^4}{\sqrt{|\xi|^2 + |x|^2 / N^4} + \sqrt{|\xi|^2 + |x_0|^2/N^4}}.
\end{split}
\end{equation*}
Under our assumptions on $x$, we have
\begin{equation*}
\big| |x|-|x_0| \big| \leq |x-x_0| \lesssim N
\end{equation*}
and the following estimate is likewise easy to see:
\begin{equation*}
\frac{|x|+ |x_0|}{\sqrt{|\xi|^2 + |x|^2 / N^4} + \sqrt{|\xi|^2 + |x_0|^2/N^4}} \lesssim N^2.
\end{equation*}
We find
\begin{equation*}
\big| t \big( \sqrt{|\xi|^2 + |x|^2/N^4} - \sqrt{|\xi|^2 + |x_0|^2 / N^4} \big) \big| \lesssim 1.
\end{equation*}
The following estimate for derivatives $\alpha \in \N_0^d$, $ |\alpha| \leq \N_0^d$ follows from similar considerations:
\begin{equation*}
\big| \partial_{\xi}^\alpha  \big( t ( \sqrt{|\xi|^2 + |x|^2/N^4} - \sqrt{|\xi|^2 + |x_0|^2 / N^4} ) \big) \big| \lesssim_{\alpha,d} 1.
\end{equation*}

We turn to the estimate
\begin{equation*}
\big| \partial_{\xi_i}^{\ell} \mathcal{E}(x',t';\xi') \big| \lesssim_{n,\ell} 1.
\end{equation*}
We shall obtain an estimate
\begin{equation}
\label{eq:MultinomialEstimate}
\big| \partial_{\xi_m}^{\ell} \big( \prod_{i=1}^k \partial_z^{\alpha_i} p(N^{-2} x', \xi') \big) \big| \leq C(\ell,k,d) \prod_{i=1}^k \partial_z^{\alpha_i} \tilde{p}(N^{-2} x', \xi')
\end{equation}
with $\tilde{p}$ denoting the analytic majorization of $p$ at $(N^{-2} x', \xi')$.

Recall the Leibniz rule for higher derivatives
\begin{equation*}
\big( \prod_{i=1}^k f_i \big)^{(\ell)} = \sum_{j_1 + \ldots + j_k = \ell} \binom{n}{j_1,\ldots,j_k} \big[ \prod_{i=1}^k \partial^{j_i} f_i \big]
\end{equation*}
denoting the multinomial coefficients as
\begin{equation*}
\binom{n}{j_1,\ldots,j_k} = \frac{n!}{j_1! \ldots j_k!}.
\end{equation*}
Now we can estimate the modulus in \eqref{eq:MultinomialEstimate} using the analytic majorization:
\begin{equation*}
\big| \partial_{\xi_m}^{\ell} \big( \prod_{i=1}^k \partial_{z}^{\alpha_i} p(N^{-2} x', \xi') \big) \big| \leq \sum_{j_1 + \ldots+ j_k = \ell} \binom{n}{j_1,\ldots,j_k} \prod_{i=1}^k \partial^{j_i}_{\xi_m} \partial_z^{\alpha_i} \tilde{p}(N^{-2} x', \xi'). 
\end{equation*}

With the analytic majorization given by
\begin{equation*}
\frac{Cr}{r - \sum_{i=1}^{2d} z_i} = C \sum_{\alpha \in \N_0^{2d}} r^{-|\alpha|} z_i^\alpha, \qquad r \geq \frac{1}{4},
\end{equation*}
we have the following estimate:
\begin{equation*}
\begin{split}
 \partial_{\xi_{m}}^{j_i} \partial_z^{\alpha_i} \tilde{p}(N^{-2} x',\xi')  &\leq (|\alpha_i|+1) \ldots (|\alpha_i| + j_i) 4^{j_i} \partial_z^{\alpha_i} \tilde{p}(N^{-2} x', \xi') \\
 &\leq 4^{j_i} (k+d+1)^{j_i} \partial_z^{\alpha_i} \tilde{p}(N^{-2} x', \xi').
\end{split}
\end{equation*}
Here we use that $|\alpha_j| \leq k$ since $\sum_{j=1}^k |\alpha_j| = 2(k-1)$ and $j_i \leq \ell \leq d+1$.

Above we used the estimate
\begin{equation*}
(|\alpha_i|+1) \ldots (|\alpha_i| + j_i) \leq (k+d+1)^{j_i},
\end{equation*}
and it follows from summing the multinomial coefficients:
\begin{equation*}
\begin{split}
&\quad \sum_{j_1 + \ldots+ j_k = \ell} \binom{\ell}{j_1,\ldots,j_k} \prod_{i=1}^k \partial_{\xi_{\ell}}^{j_i} \partial_z^{\alpha_i} \tilde{p}(N^{-2} x', \xi') \\
 &\leq \sum_{j_1+\ldots+j_k = \ell} \binom{\ell}{j_1,\ldots,j_k} \prod_{i=1}^k \big[ 4^{j_i} (|\alpha_i|+1) \ldots (|\alpha_i + j_i) \partial_z^{\alpha_i} \tilde{p}(N^{-2} x', \xi') \big] \\
&\leq \sum_{j_1+\ldots+j_k = \ell} \binom{\ell}{j_1,\ldots,j_k} 4^{\ell} (k+1+d)^{\ell} \prod_{i=1}^\ell \partial_z^{\alpha_i} \tilde{p}(N^{-2} x', \xi') \\
&\leq 4^{2 \ell} (k+1+d)^{\ell} \prod_{i=1}^\ell \partial_z^{\alpha_i} \tilde{p}(N^{-2} x', \xi').
\end{split}
\end{equation*}
Clearly, for $k \leq d+10$ we have
\begin{equation*}
4^{2d} (k+d+1)^{d+1} \leq C_d.
\end{equation*}
For $k \geq d+10$ and choosing $N \geq 2^{2d}$, we obtain
\begin{equation*}
\frac{k^{d+1}}{N^{k-2}} \leq \frac{k^{d+1}}{N^{\frac{k}{2}}} = \frac{k^{d+1}}{2^{dk}} \to 0 \text{ as } k \to \infty.
\end{equation*}

\end{proof}

By means of \eqref{eq:RegularityHigherOrder} we obtain a Fourier series expansion:
\begin{lemma}
Let $\chi \in C^\infty_c(B(0,2))$. The following expansion holds:
\begin{equation}
\label{eq:FourierSeriesHigherOrder}
e^{i \mathcal{E}_N(x,t;\xi)} \chi(\xi) = \sum_{k \in \Z^d} \alpha_k(x,t) e^{i \xi \cdot k} \chi(\xi)
\end{equation}
with the estimate
\begin{equation}
\label{eq:FourierCoefficientEstimate}
|\alpha_k(x,t)| \lesssim_d (1+|k|)^{-(d+1)}
\end{equation}
uniform in $(x,t)$.
\end{lemma}
\begin{proof}
Let $\tilde{\chi} \in C^\infty_c(B(0,2))$ denote a bump function with mildly enlarged support such that $\chi \tilde{\chi} = \chi$ and $\T = \R / (2 \pi \Z)$. We consider the periodization of $\tilde{\chi}(\xi) e^{i \mathcal{E}_N(x,t;\xi)}$, which is expanded as a Fourier series:
\begin{equation*}
\big( \tilde{\chi}(\xi) e^{i \mathcal{E}_N(x,t;\xi)} \big)_{\text{per}}  = \sum_{k \in \Z^d} \alpha_k(x,t) e^{i \xi \cdot k}
\end{equation*}
with Fourier coefficients given by
\begin{equation*}
\int_{\T^d} e^{-i \xi \cdot k } e^{i \mathcal{E}_N(x,t;\xi)} \tilde{\chi}(\xi) d \xi = \alpha_k(x,t).
\end{equation*}
The estimate \eqref{eq:FourierCoefficientEstimate} is immediate from integration by parts and Lemma \ref{lem:EstimatesRemainderTerm}.
\end{proof}

\subsection{Reduction to Klein-Gordon phase function}

Let
\begin{equation*}
\varphi_{\text{lin}}(x,t;\xi) = x \cdot \xi + t \sqrt{|\xi|^2 + |x_0|^2/N^4}.
\end{equation*}
such that
\begin{equation*}
\varphi_N(x,t;\xi) = \varphi_{\text{lin}}(x,t;\xi) + \mathcal{E}_N(x,t;\xi).
\end{equation*}
The first two terms clearly do not effect the estimates for the oscillatory integral and are thus omitted in the following.

In the next step we use the Fourier series lemma to find
\begin{equation*}
\begin{split}
\int e^{i \varphi_N(x,t;\xi)} a(x,t;\xi) \hat{u}_0(\xi) d\xi &= \int e^{i \varphi_{\text{lin}}(x,t;\xi)} e^{i \mathcal{E}_N(x,t;\xi)} a(x,t;\xi) \hat{u}_0(\xi) d\xi \\
&= \sum_{k \in \Z^d} \alpha_k(x,t) \int e^{i \varphi_{\text{lin}}(x,t;\xi)} e^{i \xi \cdot k} a(x,t;\xi) \hat{u}_0(\xi) d\xi \\
&= \sum_{k \in \Z^d} \alpha_k(x,t) \int e^{i \varphi_{\text{lin}}(x+k,t;\xi)} a(x,t;\xi) \hat{u}_0(\xi) d\xi.
\end{split}
\end{equation*}
By another Fourier series argument we can separate variables for the amplitude function:
\begin{equation*}
a(x,t;\xi) = \sum_{\ell \in \Z^d} A_{\ell}(x,t) \beta(\xi) e^{i l \cdot \xi} \text{ with } |A_{\ell}(x,t)| \lesssim (1+|\ell|)^{-(d+1)} 
\end{equation*}
and $\beta \in C^\infty_c(B_d(0,4) \backslash B_d(0,1/4))$.
Here we use $a \in S^0$ uniformly in $t$.

\medskip

Plugging this into the above we find
\begin{equation*}
\int e^{i \varphi_N(x,t;\xi)} a(x,t;\xi) \hat{u}_0(\xi) d\xi = \sum_{k,\ell \in \Z^d} \alpha_k(x,t) A_{\ell}(x,t) \int e^{i \varphi_{\text{lin}}(x+k+\ell,t;\xi)} \beta(\xi) \hat{u}_0(\xi) d\xi.
\end{equation*}

Using the decay properties of the Fourier series coefficients and translation invariance of the linearized phase function, it suffices to analyze the constant coefficient propagator:
\begin{equation*}
S_{m^2} f(x,t) = \int e^{i ( x \cdot \xi + t \sqrt{|\xi|^2 + m^2})} \chi_1(\xi) \hat{f}(\xi) d\xi.
\end{equation*}

We summarize our findings of this section in the following proposition:
\begin{proposition}
\label{prop:KleinGordonReduction}
Let $C \gg 1$ and $f = \mathfrak{P}_{\geq C} f$. Suppose that for $s \geq \frac{1}{p}$ the following estimate holds for any dyadic $m^2$:
\begin{equation}
\label{eq:KleinGordonSmoothingAssumption}
\| S_{m^2} f \|_{L^p_{t,x}(B_{d+1}(0,N))} \lesssim N^s \| f \|_{L^p}.
\end{equation}
Then it follows
\begin{equation}
\label{eq:WaveHermiteSmoothingConsequence}
\| \cos(t \sqrt{\mathcal{H}}) f \|_{L^p_{t,x}([0,1] \times \R^d)} \lesssim_\varepsilon \| f \|_{L^p_{s-\frac{1}{p}+\varepsilon}}.
\end{equation}
\end{proposition}
\begin{proof}
It suffices to show the estimate for dyadically localized functions $f$ as a consequence of Minkowki's inequality:
\begin{equation*}
\| \cos(t \sqrt{\mathcal{H}}) P_N f \|_{L^p_{t,x}([0,1] \times \R^d)} \lesssim N^{s-\frac{1}{p}} \| P_N f \|_{L^p_x(\R^d)}
\end{equation*}
with $\text{supp}(a(x,t;\cdot)) \subseteq B(0,4N) \backslash B(0,N/4)$ and $N \gg 1$ as low frequencies can be estimated trivially by the considerations of Section \ref{subsection:LowHermiteFrequencies}. Moreover, by Remark \ref{rem:LongTimeExistence}, choosing $N$ large enough, the parametrix exists for times $|t| \geq 1$. We use essentially finite speed of propagation to localize the estimate as
\begin{equation*}
\big\| e^{i t  \sqrt{\mathcal{H}}} P_N f \|_{L^p_{t,x}([0,1] \times B_d(x_0,1))} \lesssim N^{s- \frac{1}{p}} \| f \|_{L^p(\R^d)}.
\end{equation*}
We plug in the parametrix, which reduces the above to the oscillatory integral estimate:
\begin{equation*}
\big\| \int e^{i \varphi(x,t;\xi)} a(x,t;\xi) \hat{f}(\xi) d\xi \big\|_{L^p_{t,x}([0,1] \times B_d(x_0,1))} \lesssim N^{s-\frac{1}{p}} \| f \|_{L^p(\R^d)}.
\end{equation*}
Now we use the scaling $\xi = N \xi'$, $x=N^{-1} x'$, $t=N^{-1} t'$ to normalize the frequencies to the unit annulus. We have
\begin{equation}
\label{eq:RescaledParametrix}
\begin{split}
&\quad \big\| \int e^{i \varphi(x,t;\xi)} a(x,t;\xi) \hat{f}(\xi) d\xi \big\|_{L^p_{t,x}([0,1] \times B_d(x_0,1))} \\
&= N^{-\frac{d}{p}} N^{-\frac{1}{p}} N^d \big\| \int e^{i \varphi_N(x', t'; \xi')} a_N(x',t';\xi') \hat{f}(N \xi') d\xi' \big\|_{L^p_{t,x}([0,N] \times B_d(N x_0, N))}
\end{split}
\end{equation}
with $\varphi_N(x',t';\xi') = \varphi(N^{-1} x', N^{-1} t; N\xi')$ and $a_N(x',t';\xi') = a(N^{-1} x', N^{-1} t'; N \xi')$.

Define $\hat{g}(\xi') = N^{d-\frac{d}{p}} \hat{f}(N \xi')$ such that $ \| g \|_{L^p} = \| f \|_{L^p}$. We linearize the phase function $\varphi_N(x',t';\xi')$ in $(x',t')$ around $x'=Nx_0$ and $t' =0$ and use a Fourier series argument to find:
\begin{equation*}
\begin{split}
&\quad \big\| \int e^{i \varphi_N(x',t';\xi')} a_N(x',t';\xi') \hat{g}(\xi') d\xi' \big\|_{L^p_{t',x'}([0,N] \times B_d(N x_0,N))} \\
&\leq \sum_{k,\ell \in \Z^d} (1+|k|)^{-(d+1)} (1+|\ell|)^{-(d+1)} \\
&\qquad \quad \times \big\| \int e^{i \varphi_{\text{lin}}(x'+k+\ell,t';\xi')} \beta(\xi') \hat{g}(\xi') d\xi' \big\|_{L^p_{t',x'}([0,N] \times B_d(N x_0, N))}.
\end{split}
\end{equation*}
By translation invariance we can apply the hypothesis \eqref{eq:KleinGordonSmoothingAssumption} to estimate
\begin{equation*}
\| S_{m^2} g(x'+k+\ell,t') \|_{L^p_{t,x}([0,N] \times B_d(Nx_0, N))} \lesssim \| g \|_{L^p}.
\end{equation*}
Hence, we obtain
\begin{equation*}
\begin{split}
&\quad \big\| \int e^{i \varphi_N(x',t';\xi')} a_N(x',t';\xi') \hat{g}(\xi') d\xi' \big\|_{L^p_{t',x'}([0,N] \times B_d(N x_0,N))} \\
&\lesssim \sum_{k,\ell \in \Z^d} (1+|k|)^{-(d+1)} (1+|\ell|)^{-(d+1)} N^s \| g \|_{L^p} \lesssim N^s \| f \|_{L^p}.
\end{split}
\end{equation*}
Plugging this into \eqref{eq:RescaledParametrix} we complete the proof.
\end{proof}

%

\section{Necessary conditions for local smoothing and pointwise estimates}
\label{section:Knapp}

In this section we obtain necessary conditions for local smoothing estimates for Klein-Gordon equations. To this end, we test with the anisotropic and isotropic Knapp example. Finally, we compare to pointwise estimates. Define
\begin{equation}
S_{m^2,N} f(x,t) = \int e^{i ( x \cdot \xi + t \sqrt{|\xi|^2 + m^2})} \chi_N(\xi) \hat{f}(\xi) d\xi.
\end{equation}

\subsection{Necessary conditions for local smoothing estimates}

We show the following proposition:
\begin{proposition}
\label{prop:NecessaryConditions}
Let $2 \leq p < \infty$, $N \in 2^{\N_0}$, and $m \in 2^{\Z} \cup \{ 0 \}$. 

\begin{itemize}
\item[(i)]
Necessary for the estimate
\begin{equation}
\label{eq:NecessaryLocalSmoothing}
\| S_{m^2,N} f \|_{L^p_{t,x}([0,1] \times \R^d)} \lesssim N^s \| f \|_{L^p}
\end{equation}
to hold uniformly in $N$ and $m$ is 
\begin{equation}
\label{eq:LocalSmoothingEllipticNecessary}
s \geq \max ( d \big( \frac{1}{2} - \frac{1}{p} \big) - \frac{1}{p}, 0).
\end{equation} 

\item[(ii)]
Suppose that $m \leq m_0$. Necessary for \eqref{eq:NecessaryLocalSmoothing} to hold uniformly in $N \in 2^{\N_0}$ (with implicit constant depending on $m_0$) is
\begin{equation}
\label{eq:LocalSmoothingWaveNecessary}
s \geq \max \big( (d-1) \big( \frac{1}{2} - \frac{1}{p} \big) - \frac{1}{p}, 0 \big).
\end{equation}
\end{itemize}
\end{proposition}
\begin{proof}
We normalize to unit frequencies using the scaling
\begin{equation*}
\xi = N \xi', \quad x = N^{-1} x', \quad t = N^{-1} t'.
\end{equation*}
This yields
\begin{equation*}
\begin{split}
&\quad \big\| \int e^{i ( x \cdot \xi + t \sqrt{|\xi|^2+  m^2})} \hat{f}(\xi) \chi_N(\xi) d\xi \big\|_{L^p_{t,x}([0,1] \times \R^d)} \\
&= N^{d-\frac{d}{p}} N^{-\frac{1}{p}} \big\| \int e^{i ( x' \cdot \xi' + t' \sqrt{|\xi'|^2 + m^2/N^2})} \hat{f}(N \xi') \chi_1(\xi') d\xi' \big\|_{L^p_{t,x}([0,N] \times \R^d)}.
\end{split}
\end{equation*}
Define $g: \R^d \to \C$ via $\hat{g}(\xi') = N^{d-\frac{d}{p}} \hat{f}(N \xi')$ such that $\| g \|_{L^p(\R^d)} = \| f \|_{L^p(\R^d)}$ and let $\mu = m/N$.
We turn to the analysis of
\begin{equation*}
\big\| \int e^{i ( x' \cdot \xi' + t' \sqrt{|\xi'|^2 + \mu^2})} \chi_1(\xi') \hat{g}(\xi') d\xi' \big\|_{L^p_{t,x}([0,N] \times \R^d)} \lesssim N^{\bar{s}} \| g \|_{L^p}.
\end{equation*}
The above estimate holds with $\bar{s}$ if and only if the estimate \eqref{eq:NecessaryLocalSmoothing} holds with $s = \bar{s}- \frac{1}{p}$.

Note the following two extreme cases: If $\mu^2 \lesssim 1/N$, we have
\begin{equation}
\label{eq:WaveReduction}
t' \sqrt{|\xi'|^2 + \mu^2} = t' |\xi'| + t' (\sqrt{|\xi'|^2+\mu^2} - |\xi'|).
\end{equation}
If $\mu^2 \lesssim 1/N$, we have
\begin{equation*}
\big| \frac{t' \mu^2}{|\xi'| + \sqrt{|\xi'|^2 + \mu^2}} \big| \lesssim 1 \text{ for } |t'| \ll N.
\end{equation*}
This means for $\mu^2 \lesssim \frac{1}{N} \Leftrightarrow m^2 \lesssim N$ the evolution resembles the wave propagation. This can be made precise, e.g., by a Fourier series argument as used in the previous section when linearizing the phase function. Testing with the anisotropic and isotropic Knapp examples for the wave equation, which will be recalled in a more general context below, we find the necessary conditions \eqref{eq:LocalSmoothingWaveNecessary}. 

This shows (ii) since the low frequencies can be trivially estimated. Choosing $N$ large enough such that $m_0^2 \lesssim N$, we can use the observation in \eqref{eq:WaveReduction} to reduce to the wave evolution.

\medskip
Now consider $\mu \gg N \Leftrightarrow m \gg N^2$. In this case we can write
\begin{equation*}
t' \sqrt{|\xi'|^2 + \mu^2} = t' \mu + t' (\sqrt{|\xi'|^2 + \mu^2} - \mu)
\end{equation*}
which shows that for $\mu \gg N$ we have
\begin{equation}
\label{eq:ComparisonWave}
t'(\sqrt{|\xi'|^2 + \mu^2} - \mu) = \frac{t' |\xi'|^2}{\sqrt{|\xi'|^2 + \mu^2} + \mu} \ll 1
\end{equation}
provided that $|t'| \ll N$. The factor $e^{it' \mu}$ is a pure phase and $e^{it'( \sqrt{|\xi|^2 + \mu^2} - \mu)}$ can be expanded into Fourier series or exponential sums to see that it does not significantly contribute.

This means there is no significant change in the $L^p$-norm for $|t'| \ll N$. ``Local smoothing" holds with $\bar{s}=\frac{1}{p} \Leftrightarrow s = 0$.

\bigskip

We turn to the intermediate cases
\begin{equation*}
\frac{1}{\sqrt{N}} \lesssim \mu \lesssim N.
\end{equation*}

Firstly, consider the anisotropic Knapp example: The anisotropic Knapp example is supposed to linearize the phase function with small Fourier support such that
\begin{equation}
\label{eq:TransportKnapp}
\big| \int e^{i ( x \cdot \xi + t \sqrt{|\xi|^2 + \mu^2})} \hat{g}(\xi) d\xi \big| \approx \big| g(x+ \frac{t \xi_0}{\sqrt{|\xi_0|^2 + \mu^2}} ) \big|.
\end{equation}
Note that we can choose $\text{supp}(\hat{g}) \subseteq B(\xi_0,N^{-\frac{1}{2}})$ independently of $\mu$ to achieve the above because $\varphi \in C^2$.

More precisely, let $\eta_{\alpha} \in C^\infty_c((-2 \alpha, 2 \alpha))$ be a smoothed version of the indicator function on $[-\alpha, \alpha]$. Define
\begin{equation*}
\hat{g}(\xi) = \eta_{N^{-\frac{1}{2}}}(\xi_1-1) \eta_{N^{-\frac{1}{2}}}(\xi_2) \ldots \eta_{N^{-\frac{1}{2}}}(\xi_d).
\end{equation*}
We can carry out a Taylor expansion
\begin{equation*}
\begin{split}
&\quad t \sqrt{|\xi|^2 + \mu^2} = t \big( \sqrt{|\xi_0|^2 + \mu^2} + \frac{\xi_0}{\sqrt{|\xi_0|^2 + \mu^2}} \cdot (\xi - \xi_0) +\mathcal{E}(\xi_0,\xi) \big) \\
 &\text{ with } \mathcal{E}(\xi_0,\xi) = O(|\xi-\xi_0|^2).
 \end{split}
\end{equation*}
Hence, the error term does not significantly contribute to oscillations for $|t| \lesssim N$.

\medskip

It turns out that in the case $\mu \not\sim 1$ this is not the maximal size of the Fourier support of $g$ we can allow such that \eqref{eq:TransportKnapp} remains true. To see this we compute the principal curvatures of $\varphi$. We have
\begin{equation*}
\partial^2_{ij} \varphi = \frac{(|\xi|^2 + \mu^2) \delta_{ij} - \xi_i \xi_j}{(|\xi|^2 + \mu^2)^{\frac{3}{2}}}.
\end{equation*}
Let $\xi'= O \xi = |\xi| e_1$ be the rotation to rotate $\xi$ into $e_1$-direction. We obtain for the conjugation of $\partial^2 \varphi$ with $O$:
\begin{equation*}
O^t \partial^2 \varphi O = \frac{(|\xi|^2 + \mu^2) \delta_{ij} - |\xi|^2 \delta_{i1}}{(|\xi|^2 + \mu^2)^{\frac{3}{2}}}.
\end{equation*}
This shows that the principal curvature in the radial direction is of size $\sim \frac{\mu^2}{(|\xi|^2 + \mu^2)^{\frac{3}{2}}} \sim \mu^2 \wedge \frac{1}{\mu}$. In the angular direction we find the principal curvatures to be of size $\sim \frac{1}{\mu} \wedge 1$.

First, we consider the case $\mu \lesssim 1$, in which case the principal curvature into radial direction is of size $\mu^2$ and into angular directions of size $\sim 1$. This suggests to choose the Knapp example to have support of size $N^{-\frac{1}{2}}/\mu$ into the radial direction and $N^{-\frac{1}{2}}$ into the angular direction.

Indeed, consider $\xi_1 \in [1,1+N^{-\frac{1}{2}} \mu^{-1}]$ and $\xi_j \in [-N^{-\frac{1}{2}},N^{-\frac{1}{2}}]$ for $j=2,\ldots,d$. We use the anisotropic dilation $\xi_1' = \mu \xi_1$ to obtain the phase function
\begin{equation*}
\psi(\xi') = \sqrt{|\xi_1'|^2 / \mu^2 + |\xi_2'|^2 + \ldots + |\xi_d'|^2 + \mu^2}.
\end{equation*}
The support properties are
\begin{equation}
\label{eq:SupportAnisotropicDilation}
\xi_1' \in [\mu, \mu + N^{-\frac{1}{2}}], \quad \xi_j' \in [-N^{-\frac{1}{2}},N^{-\frac{1}{2}}] \text{ for } j=2,\ldots,d.
\end{equation}
It is straight-forward to compute that
\begin{equation*}
\partial^2_{\xi' \xi'} \psi(\xi') = 1 + O ( N^{-\frac{1}{2}} \mu^{-1})  \text{ for } \xi' \text{ satisfying } \eqref{eq:SupportAnisotropicDilation}.
\end{equation*}
Since we suppose that $\mu \gg N^{-\frac{1}{2}}$, this makes $\{ (\xi',\psi(\xi')) : \xi' \text{ like in } \eqref{eq:SupportAnisotropicDilation} \}$ a uniformly elliptic surface.

In case $\mu \gg 1$ we can choose the support
\begin{equation*}
\xi_1 \in [1,1+N^{-\frac{1}{2}} \mu^{\frac{1}{2}}], \quad \xi_j \in [-N^{-\frac{1}{2}} \mu^{\frac{1}{2}}, N^{-\frac{1}{2}} \mu^{\frac{1}{2}} ].
\end{equation*}
We use the isotropic dilation $\xi_i' = \xi_i / \mu^{\frac{1}{2}}$ to find the phase function
\begin{equation*}
\psi(\xi') = \sqrt{\mu |\xi_1'|^2 + \mu |\xi_2'|^2 + \ldots + \mu |\xi_d'|^2 + \mu^2}
\end{equation*}
with the support
\begin{equation}
\label{eq:SupportAnisotropicDilationII}
\xi_1' \in [\mu^{-\frac{1}{2}}, \mu^{-\frac{1}{2}} + N^{-\frac{1}{2}}] , \quad \xi_j' \in [- N^{-\frac{1}{2}}, N^{-\frac{1}{2}}] \text{ for } j=2,\ldots,d,
\end{equation}
and we obtain
\begin{equation*}
\partial_{\xi' \xi'}^2 \psi(\xi') = 1 + O(N^{-1} ) \text{ for } \xi' \text{ satisfying } \eqref{eq:SupportAnisotropicDilationII}.
\end{equation*}
This shows uniform ellipticity provided that $\mu \ll N$.

\medskip

We turn to the isotropic Knapp example: In this case we consider the initial data (for the unrescaled problem)
\begin{equation*}
\hat{f}(\xi) = \theta(N^{-1} \xi) e^{-i \sqrt{|\xi|^2 + m^2}},
\end{equation*}
where $\theta \in C^\infty_c(B(0,4) \backslash B(0,1/4))$ denotes a radially symmetric bump function.

We have
\begin{equation*}
\begin{split}
&\quad \| S_{m^2} f \|_{L^p_{t,x}([0,1] \times \R^d)} \\
&= N^d \big\| \int e^{i (\langle N x, \xi' \rangle + t N \sqrt{|\xi'|^2 + m^2/N^2} - N \sqrt{|\xi'|^2 + m^2 / N^2})} \theta(\xi') d \xi' \big\|_{L^p([0,1] \times \R^d)}.
\end{split}
\end{equation*}
For $N \lesssim m \lesssim N^2$ we have
\begin{equation*}
|(t-1) N \sqrt{|\xi'|^2 + m^2 / N^2}| \lesssim 1 \text{ for } |t-1| \lesssim 1/N.
\end{equation*}

Hence, we estimate
\begin{equation*}
N^d \big\| e^{i ( \langle N x, \xi' \rangle + (t-1) N \sqrt{|\xi'|^2  + m^2 /N^2})} \theta(\xi') d \xi' \big\|_{L^p([0,1] \times \R^d)} \gtrsim N^{d - \frac{1}{p} - \frac{d}{p}}
\end{equation*}
because
\begin{equation*}
\big\| \int e^{i \langle N x, \xi' \rangle} \theta(\xi') d\xi' \big\|_{L^p(\R^d)} \gtrsim N^{-\frac{d}{p}}.
\end{equation*}

We compute the $L^p$-norm of the initial data
\begin{equation*}
\begin{split}
&\quad \big\| \int e^{i \langle x, \xi \rangle} \theta(N^{-1} \xi) e^{-i \sqrt{|\xi|^2+ m^2}} d \xi \big\|_{L^p(\R^d)} \\
&= N^{d-\frac{d}{p}} \big\| \int e^{i \langle x', \xi' \rangle} e^{-i N \sqrt{|\xi'|^2 + \mu^2}} \theta(\xi') d\xi' \big\|_{L^p(\R^d)}
\end{split}
\end{equation*}
via the method of (non-)stationary phase.

We consider the case $\mu \lesssim 1$ first. 
We find the stationary points from the first derivative:
\begin{equation*}
x - \frac{N \xi}{\sqrt{|\xi|^2 + \mu^2}} = 0.
\end{equation*}
Consequently, we have stationary points for $|x| \sim N$.

Above we computed the eigenvalues of the Hessian of the phase function
\begin{equation*}
\varphi_N(\xi') = N \sqrt{|\xi'|^2 + \mu^2}
\end{equation*}
to be $N(\mu^2 \wedge \frac{1}{\mu})$ in the radial direction and $N \big( \frac{1}{\mu} \wedge 1 \big)$ in the angular direction.
For $|x| \sim N$ we obtain the asymptotic
\begin{equation*}
|F(x)| \sim (N \mu^2)^{-\frac{1}{2}} \underbrace{ N^{-\frac{1}{2}} \ldots N^{-\frac{1}{2}}}_{d-1 \text{ times}} = \mu^{-1} N^{-\frac{d}{2}}
\end{equation*}
from the Van der Corput-lemma \cite{Stein1993}.

Taking the rapid decay of the oscillatory integral for $|x| \ll N$ into account, we obtain
\begin{equation*}
N^d N^{-\frac{d}{p}} \big\| \int e^{i ( \langle x', \xi' \rangle - N \sqrt{|\xi'|^2 + \mu^2})} \theta(\xi') d\xi' \big\|_{L^p(\R^d)} \lesssim N^d N^{-\frac{d}{p}} N^{\frac{d}{p}} \mu^{-1} N^{-\frac{d}{2}} \sim N^{\frac{d}{2}} \mu^{-1}.
\end{equation*}
For the estimate \eqref{eq:NecessaryLocalSmoothing} to hold we obtain the condition
\begin{equation*}
N^d N^{-\frac{1}{p}} N^{-\frac{d}{p}} \lesssim N^{s+\frac{d}{2}} \mu^{-1}.
\end{equation*}
Consequently, for $\mu \sim 1$ we obtain the condition
\begin{equation*}
s \geq d \big( \frac{1}{2} - \frac{1}{p} \big) - \frac{1}{p}.
\end{equation*}
We remark that for $\mu \gg 1$ the necessary condition actually becomes weaker because the dispersive properties attenuate.
\end{proof}
\subsection{Pointwise estimates}

For reference we remark on pointwise estimates:
\begin{equation}
\label{eq:PointwiseEstimateKleinGordon}
\| S_{m^2,N} f(1) \|_{L^p_x(\R^d)} \lesssim N^s \| f \|_{L^p_x(\R^d)}.
\end{equation}
A variant of the computation from Proposition \ref{prop:NecessaryConditions} shows that $s \geq d \big| \frac{1}{2} - \frac{1}{p} \big|$ is necessary for \eqref{eq:PointwiseEstimateKleinGordon} to hold uniformly in $N \in 2^{\N_0}$ and $m \in 2^{\Z}$. Details are omitted to avoid repetition. For reference we prove the estimate in the following:
\begin{proposition}[Pointwise~estimate~for~Klein-Gordon~equations]
\label{prop:PointwiseEstimates}
Let $2 \leq p \leq \infty$. The estimate \eqref{eq:PointwiseEstimateKleinGordon} holds with constant independent of $m \in 2^{\Z}$ and $N \in 2^{\N_0}$ for
\begin{equation*}
s \geq d \big( \frac{1}{2} - \frac{1}{p} \big).
\end{equation*}
\end{proposition}
\begin{proof}
For $p=2$ the estimate is immediate from Plancherel's theorem. Let $p > 2$ in the following. We rescale the frequencies to the unit annulus and rescale the spatial variable dually:
\begin{equation*}
\begin{split}
&\quad \big\| \int_{|\xi| \sim N} e^{i(x \cdot \xi + \sqrt{|\xi|^2 + m^2})} \hat{f}(\xi) d\xi \big\|_{L^p(\R^d)} \\
 &= N^{d-\frac{d}{p}} \big\| \int_{|\xi'| \sim 1} e^{i(x' \cdot \xi' + N \sqrt{|\xi'|^2 + m^2/N^2})} \hat{f}(N \xi') d\xi' \big\|_{L^p(\R^d)}. 
 \end{split}
\end{equation*}
We define $\hat{g}(\xi') = N^{d - \frac{d}{p}} \hat{f}(N \xi')$ such that $\| g \|_{L^p(\R^d)} = \| f \|_{L^p(\R^d)}$. Let $m_0^2 = m^2/N^2$ to ease notation. Let $\theta \subseteq B(0,2)$ be a ball of size $N^{-\frac{1}{2}}$. We observe that independently of $m_0$ we have the following $L^p$-bound as consequence of a kernel estimate:
\begin{equation*}
\big\| \int_{\theta} e^{i(x' \cdot \xi' + N \sqrt{|\xi'|^2 + m_0^2})} \hat{g}(\xi') d\xi' \big\|_{L^p(\R^d)} \lesssim \| P_{\theta} g \|_{L^p(\R^d)}.
\end{equation*}
Above $P_{\theta}$ denotes the smooth Fourier projection to a mildly enlarged set $\tilde{\theta}$. Hence, by Minkowksi's inequality, we find
\begin{equation*}
\big\| \int_{|\xi'| \sim 1} e^{i(x' \cdot \xi' + N \sqrt{|\xi'|^2 + m_0^2})} \hat{g}(\xi') d\xi' \big\|_{L^p(\R^d)} \leq \sum_{\theta: N^{-\frac{1}{2}}-\text{ball}} \| P_{\theta} g \|_{L^p(\R^d)}.
\end{equation*}
Above the sum is taken over an essentially disjoint cover of the unit annulus with $N^{-\frac{1}{2}}$-ball.
For $p=\infty$, we obtain from H\"older's inequality:
\begin{equation*}
\sum_{\theta: N^{-\frac{1}{2}}-\text{ball}} \| P_{\theta} g \|_{L^p(\R^d)} \lesssim N^{\frac{d}{2}} \sup_{\theta} \|P_{\theta} g \|_{L^\infty(\R^d)} \lesssim N^{\frac{d}{2}} \| g \|_{L^\infty(\R^d)}.
\end{equation*}
The ultimate estimate following from another kernel estimate. This concludes the proof for $p=\infty$. For $2 < p <\infty$ the claim follows from interpolation.
\end{proof}

\section{Proof of local smoothing estimates in one dimension}
\label{section:ProofLocalSmoothing1d}

Let $u$ be a solution to
\begin{equation*}
\left\{ \begin{array}{cl}
\partial_t^2 u &= \Delta u - |x|^2 u, \quad (t,x) \in \R \times \R, \\
u(0) &= u_0 \in \mathcal{S}(\R), \quad \dot{u}(0) = 0.
\end{array} \right.
\end{equation*}

In this section we show for $s>0$
\begin{equation}
\label{eq:LocalSmoothing1d}
\| u \|_{L^4_{t,x}([0,1] \times \R)} \lesssim_s \| u_0 \|_{L^4_{s}(\R)}. 
\end{equation}
This yields the proof of Theorem \ref{thm:LocalSmoothing1d} by interpolation with the trivial estimate
\begin{equation*}
\big\| u \|_{L^2_{t,x}([0,1] \times \R)} \lesssim \| u_0 \|_{L^2(\R)}
\end{equation*}
and taking $q$ large enough
\begin{equation*}
\big\| u \|_{L^q_{t,x}([0,1] \times \R)} \lesssim \| u_0 \|_{L^q_{s_q+\varepsilon}(\R)}
\end{equation*}
with $s_q = \frac{1}{2} - \frac{1}{q}$. The latter estimate follows from linearization of the parametrix and pointwise estimates due to Proposition \ref{prop:PointwiseEstimates}. In the remainder of this section we are concerned with the proof of \eqref{eq:LocalSmoothing1d}. 

\medskip

Invoking Proposition \ref{prop:KleinGordonReduction} the above amounts to showing
 \begin{equation}
\label{eq:OneDimLocalSmoothing}
\big\| \int_{ \{ \frac{1}{2} \leq |\xi| \leq 2 \}} e^{i (x \cdot \xi + t \sqrt{|\xi|^2 + m^2})} \hat{f}(\xi) d\xi \big\|_{L^4_{t,x}(B_2(0,N))} \lesssim_\varepsilon N^{\frac{1}{4}+\varepsilon} \| f \|_{L^4(\R)}
\end{equation}
for dyadic $\frac{1}{N} \leq m^2 \leq N$. In the following write for brevity
\begin{equation*}
S_{m^2}f(t) = e^{it \sqrt{-\Delta + m^2}} f \text{ and } \mathcal{E}_{m^2} f(x,t) = \int_{ \{ \frac{1}{2} \leq \xi \leq 2 \} } e^{i( x \cdot \xi + t \sqrt{|\xi|^2 + m^2})} f(\xi) d\xi.
\end{equation*}

In the extreme cases $m^2 \lesssim \frac{1}{N}$ and $m^2 \gtrsim N$, there is no dispersion and 
\begin{equation*}
\| S_{m^2} f(t) \|_{L^4_x(\R)} \sim \| f \|_{L^4(\R)}.
\end{equation*}

\medskip

Recall that for $m \sim 1$ we can use the Córdoba--Fefferman square function estimate (\cite{Fefferman1973,Cordoba1977}) to obtain
\begin{equation*}
\begin{split}
&\quad \big\| w_{B_2(0,N)} \int e^{i (x \cdot \xi + t \sqrt{|\xi|^2 + m^2})} \hat{f}(\xi) d\xi \big\|_{L^4_{t,x}(\R^2)} \\ &\lesssim \big\| \big( \sum_{\theta : N^{-\frac{1}{2}} - \text{interval}} \big| \int_\theta e^{i( x \cdot \xi + t \sqrt{|\xi|^2+ m^2})} \hat{f}(\xi) d\xi \big|^2 \big)^{\frac{1}{2}} \big\|_{L^4_{t,x}(\R^2)}.
\end{split}
\end{equation*}

To generalize the claim to dyadic $\frac{1}{N} \leq m^2 \leq 1$, we use a square function estimate which takes into account that the curvature is $\sim m^2$. This follows from rescaling:
\begin{proposition}
\label{prop:SquareFunction1d}
Let $N \in 2^{\N_0}$.
For $\frac{1}{N} \leq m^2 \leq 1$ the following square function estimate holds:
\begin{equation*}
\begin{split}
&\quad \big\| \int_{\{ \frac{1}{2} \leq \xi \leq 2 \}} e^{i (x \cdot \xi + t \sqrt{|\xi|^2 + m^2})} f(\xi) d\xi \big\|_{L^4_{t,x}(w_{B_2(0,N)})} \\
&\lesssim \big\| \big( \sum_{\theta : N^{-\frac{1}{2}}/m - \text{interval}} \big| \int_\theta e^{i( x \cdot \xi + t \sqrt{|\xi|^2+ m^2})} f(\xi) d\xi \big|^2 \big)^{\frac{1}{2}} \big\|_{L^4_{t,x}(w_{B_2(0,N)})}.
\end{split}
\end{equation*}
For $1 \leq m^2 \leq N$ the following square function estimate holds:
\begin{equation*}
\begin{split}
&\quad \big\| \int_{\{ \frac{1}{2} \leq \xi \leq 2 \}} e^{i (x \cdot \xi + t \sqrt{|\xi|^2 + m^2})} f(\xi) d\xi \big\|_{L^4_{t,x}(w_{B_2(0,N)})} \\
&\lesssim \big\| \big( \sum_{\theta : m^{\frac{1}{2}}
N^{-\frac{1}{2}} - \text{interval}} \big| \int_\theta e^{i( x \cdot \xi + t \sqrt{|\xi|^2+ m^2})} f(\xi) d\xi \big|^2 \big)^{\frac{1}{2}} \big\|_{L^4_{t,x}(w_{B_2(0,N)})}.
\end{split}
\end{equation*}
\end{proposition}
\begin{proof}
We consider the case $\frac{1}{N} \leq m^2 \leq 1$ first.
The key observation is that $F= w_{B_2(0,N)} \cdot \mathcal{E}_{m^2} f$ has space-time Fourier transform in the $1/N$-neighbourhood of the curve $\{(\xi,\sqrt{|\xi|^2 + m^2}) : \, \frac{1}{2} \leq \xi \leq 2 \}$. Cover the neighbourhood with finitely overlapping rectangles $\theta$ of size $N^{-\frac{1}{2}} / m \times N^{-1}$ with the long side pointing into the tangential direction and short side into normal direction.

 For $m \sim 1$, using the curvature $\sim 1$ of the curve, it is a geometric observation referred to as \emph{biorthogonality} going back to Fefferman \cite{Fefferman1973} that
\begin{equation*}
\int_{\R^2} |F|^4 \lesssim \int \big( \sum_{\theta} |F_{\theta}|^2 \big)^2.
\end{equation*}

Secondly, for $m^2 \sim \frac{1}{N}$ there is nothing to show (since there are only $O(1)$ rectangles $\theta$).

We reduce the general case of $\frac{1}{N} \ll m^2 \ll 1$ to the case of $m \sim 1$ by rescaling. Note that the curve
\begin{equation*}
\{ (\xi, \sqrt{|\xi|^2/m^2 + m^2}) : \, \xi \in [m/2,3m/2] \}
\end{equation*}
has curvature $\sim 1$. So, we can apply the Córdoba--Fefferman square function estimate after a change of variables (the extension to the space-time ellipse $E_{m,N}$ of size $N/m \times N$ is immediate):
\begin{equation*}
\begin{split}
&\quad \big\| \int_{\xi \in [\frac{1}{2},\frac{3}{2}]} e^{i (x \cdot \xi + t \sqrt{|\xi|^2 + m^2})} f(\xi) d\xi \big\|_{L^4_{t,x}(w_{B_2(0,N)})} \\
&\lesssim m^{\frac{1}{4}-1} \big\| \int_{\xi \in [\frac{m}{2},\frac{3m}{2}]} e^{i( x' \cdot \xi' +t \sqrt{|\xi'|^2/m^2 + m^2})} f(\xi'/m) d \xi' \big\|_{L^4(w_{E_{m,N}})} \\
&\lesssim m^{\frac{1}{4}-1} \big\| \big( \sum_{\tilde{\theta}: N^{-\frac{1}{2}}-\text{interval}} \big| \int_{\tilde{\theta}} e^{i( x' \cdot \xi' +t \sqrt{|\xi'|^2/m^2 + m^2})} f(\xi'/m) d \xi' \big|^2 \big)^{\frac{1}{2}} \big\|_{L^4(w_{E_{m,N}})}
\end{split}
\end{equation*}
Now we can reverse the change of variables to find
\begin{equation*}
\lesssim \big\| \big( \sum_{\theta: N^{-\frac{1}{2}}/m-\text{interval}} \big| \int_{\theta} e^{i( x \cdot \xi +t \sqrt{|\xi|^2 + m^2})} f(\xi) d \xi \big|^2 \big)^{\frac{1}{2}} \big\|_{L^4(w_{B_2(0,N)})}.
\end{equation*}
The proof for $m \leq 1$ is complete. We turn to the easier case of $1 \leq m^2 \leq N$. In this case we carry out a change of variables $\frac{\xi}{m} = \xi'$, $d \xi = m d\xi'$ and dually $t' = tm$, $xm = x'$ to find
\begin{equation*}
\begin{split}
&\quad \big\| \int_{ \xi \in [\frac{1}{2},2] } e^{i( x \cdot \xi + t \sqrt{|\xi|^2 + m^2})} f(\xi) d\xi \big\|_{L^4_{t,x}(B_2(0,N))} \\
&\quad = m^{-\frac{1}{2}} \big\| \int_{\xi' \in [m^{-1}/2,2m^{-1}]} e^{i ( x' \cdot \xi' + t' \sqrt{|\xi'|^2 + 1})} f(m \xi') m d\xi' \big\|_{L^4_{t',x'}(B_2(0,mN))}.
\end{split}
\end{equation*}
We use again that the curve $\{(\xi',\sqrt{|\xi'|^2 + 1}) : \xi' \in [m^{-1}/2,2m^{-1}] \}$ has curvature comparable to $1$. The Córdoba--Fefferman square function estimate yields
\begin{equation*}
\begin{split}
&\quad \big\| \int_{\xi' \in [m^{-1},2m^{-1}]} e^{i ( x' \cdot \xi' + t' \sqrt{|\xi'|^2 + 1})} f(m \xi') m d\xi' \big\|_{L^4_{t',x'}(B_2(0,mN))} \\
&\lesssim \big\| \big( \sum_{\theta': (mN)^{-\frac{1}{2}}-\text{interval}} \int_{\theta'} e^{i ( x' \cdot \xi' + t' \sqrt{|\xi'|^2 + 1})} f(m \xi') d\xi' \big)^{\frac{1}{2}} \big\|_{L^4_{t',x'}(w_{B_2(0,mN)})}.
\end{split}
\end{equation*}
Reversing the change of variables yields a square function estimate into intervals of length $m^{\frac{1}{2}} N^{-\frac{1}{2}}$. The proof is complete. 
\end{proof}

\begin{remark}
Also in higher dimensions the case of large $m \geq 1$, corresponding to a uniformly degenerate surface, can be handled through a simple isotropic change of variables.
\end{remark}

\emph{Proof of Theorem~\ref{thm:LocalSmoothing1d}, ctd.}
Let $(\chi_k)_{k \in \Z} \subseteq \mathcal{S}(\R)$ be a partition of unity with $\chi_k(x) = \chi_0(x-N^{\frac{1}{2}} k)$ and $\text{supp}(\hat{\chi}_0) \subseteq B(0,N^{-\frac{1}{2}})$. The family $(\chi_k)$ exists by the Poisson summation formula. We turn to the estimate of
\begin{equation*}
\| e^{it \sqrt{-\Delta + m^2}} P_1 f \|_{L^4_{t,x}(B_2(0,N))}^2 = \big\| \int_{ \{ \frac{1}{2} \leq \xi \leq 2 \} } e^{i( x \cdot \xi + t \sqrt{|\xi|^2 + m^2})} \hat{f}(\xi) d\xi \big\|^2_{L^4_{t,x}(B_2(0,N))}.
\end{equation*}
Applying the square function estimate from Proposition \ref{prop:SquareFunction1d} shows that
\begin{equation*}
\begin{split}
&\lesssim \big[ \int w_{B_2(0,N)} \big| \sum_\theta \big| \int e^{i(x \cdot \xi + t \sqrt{|\xi|^2 + m^2})} \chi_{\theta}(\xi) \hat{f}(\xi) d\xi \big|^2 \big|^2 dx dt \big]^{\frac{1}{2}} \\
&\lesssim \sup_{\substack{\| g \|_{L^2} = 1, \\ g \geq 0}} \int \sum_{\theta} \big| \int e^{i(x \cdot \xi + t \sqrt{|\xi|^2 + m^2})} \chi_{\theta}(\xi) \hat{f}(\xi) d \xi \big|^2 w_{B_2(0,N)} g(x,t) dx dt.
\end{split}
\end{equation*}
We use the decomposition
$ f = \sum_{k \in \Z} \chi_k f = \sum_k f_k.$
Then we obtain
\begin{equation*}
\int e^{i( x \cdot \xi + t \sqrt{|\xi|^2 + m^2})} \chi_{\theta}(\xi) \hat{f}(\xi) d\xi = \sum_{k \in \Z} S_{m^2,\theta} f_k.
\end{equation*}
We have
\begin{equation}
\label{eq:WavePacketOrthogonality}
\big| \sum_{k \in \Z} S_{m^2,\theta} f_k \big|^2 \lesssim_\varepsilon N^\varepsilon \sum_{k \in \Z} \big| S_{m^2,\theta} f_k \big|^2 * \phi_{N^{\frac{1}{2}}} + \text{RapDec}(N) \| f \|_{L^4}^2
\end{equation}
where $\phi_{N^{\frac{1}{2}}}$ denotes an $L^1$-normalized weight adapted to $B(0,N^{\frac{1}{2}})$.
For the proof of \eqref{eq:WavePacketOrthogonality} note that we have the kernel estimate for
\begin{equation*}
S_{m^2,\theta} f_k(x) = \int K_{m^2,\theta}(x,t;y) f_k(y) dy:
\end{equation*}
We see via integration by parts that it holds
\begin{equation*}
K_{m^2,\theta}(x,t;y) \leq C_M (1+N^{-\frac{1}{2}} |x + \frac{t \theta}{\sqrt{\theta^2 + m^2}} - y \big| \big)^{-M}
\end{equation*}
for any $M \in \N$.

Then \eqref{eq:WavePacketOrthogonality} follows from the initial localization of $f_k$. Hence, it suffices to estimate
\begin{equation*}
\sup_{\substack{\| g \|_{L^2} = 1, \\ g \geq 0}} \int \sum_{\theta,k} \big| S_{m^2,\theta} f_k \big|^2 * \phi_{N^{\frac{1}{2}}} \cdot w_{B_2(0,N)} g(x,t) dx dt.
\end{equation*}
We can furthermore dominate
\begin{equation*}
|S_{m^2,\theta} f_k |^2(x) \lesssim \big| (f_k)_\theta (x+ \frac{t \theta}{\sqrt{\theta^2 + m^2}}) \big|^2 * \phi_{N^{\frac{1}{2}}}.
\end{equation*}
By a change of variables we obtain
\begin{equation*}
\begin{split}
&\quad \sup_{\substack{ \| g \|_{L^2} = 1, \\ g \geq 0}} \int \sum_{\theta,k} | (f_k)_\theta (x)|^2 \big( w_{B_2(0,N)} g * \phi_{N^{\frac{1}{2}}} \big)(x-\frac{t \theta}{\sqrt{\theta^2 + m^2}}, t) dt dx \\
&\lesssim \sup_{\substack{ \| g \|_{L^2} = 1, \\ g \geq 0}} \int \sum_{\theta} |f_\theta|^2 \sup_{\theta} \int (\tilde{g} * \phi_{N^{\frac{1}{2}}} )(x- \frac{t \theta}{\sqrt{\theta^2 + m^2}},t) dt dx.
\end{split}
\end{equation*}
Finally, we apply the Cauchy-Schwarz inequality to summarize our findings as
\begin{equation}
\label{eq:ReductionMaximalFunction}
\begin{split}
&\quad \| S_{m^2} f \|_{L^4(B_2(0,N))}^2 \lesssim \big( \int_{\R^2} \big( \sum_\theta |f_\theta|^2 \big)^2 dx \big)^{\frac{1}{2}} \\
&\quad \quad \times \sup_{\substack{ \| g \|_{L^2} = 1, \\ g \geq 0}} \big( \int_{\R} \big| \sup_{\theta} \int_{B_1(0,N)} \big( \tilde{g} * \phi_{N^{\frac{1}{2}}} \big) \big( x- \frac{t \theta}{\sqrt{\theta^2 + m^2}},t \big) dt \big|^2 dx \big)^{\frac{1}{2}}.
\end{split}
\end{equation}

The second term is dominated by a multiple of the Kakeya maximal function with eccentricity $N^{-\frac{1}{2}}$. Indeed, the integral over $t$ and the convolution with $\phi_{N^{\frac{1}{2}}}$ can be interpreted as weighted integral over the tube of size $N^{\frac{1}{2}} \times N$ with long side determined by $\theta$. Consequently,
\begin{equation*}
\sup_{\theta} \big| \int_{B_1(0,N)} \big( \tilde{g} * \phi_{N^{\frac{1}{2}}} \big)(x-\frac{t \theta}{\sqrt{\theta^2+ m^2}}, t) dt \big| \lesssim N \sup_{\substack{(x,0) \in T, \\ T \in \mathbb{T}_{N^{-\frac{1}{2}}}}} \frac{1}{|T|} \int_T |\tilde{g}(y)| dy.
\end{equation*}
Here $\mathbb{T}_{N^{-\frac{1}{2}}}$ denotes the tubes with size $N^{\frac{1}{2}} \times N$. We denote 
\begin{equation*}
\mathcal{M} F(x,y) = \sup_{\substack{(x,0) \in T , \\ T \in \mathbb{T}_{N^{-\frac{1}{2}}}}} \frac{1}{|T|} \int_T |F(y)| dy.
\end{equation*}
Observe that $\mathcal{M} F(x,y) \sim \mathcal{M} F(x',y')$ provided that $|(x,y)-(x',y')| \lesssim N^{\frac{1}{2}}$.
Secondly, let $\tilde{\mathbb{T}}_{N^{-\frac{1}{2}}}$ denote all tubes with eccentricity $N^{-\frac{1}{2}}$ and $\tilde{\mathcal{M}}$ the corresponding maximal function. We recall the following maximal function estimate:

\begin{theorem}[{\cite[Theorem~1.1]{Cordoba1977}}]
The following estimate holds:
\begin{equation*}
\| \tilde{\mathcal{M}} F \|_{L^2(\R^2)} \lesssim \log(N)^2 \| F \|_{L^2(\R^2)}.
\end{equation*}
\end{theorem}

\begin{proof}[Proof~of~Theorem~\ref{thm:LocalSmoothing1d}]
We can now conclude
\begin{equation}
\label{eq:MaximalFunctionConsequence}
\begin{split}
\int_{\R} \big| \sup_{\theta} \int_{B_1(0,N)} \big( \tilde{g} * \phi_{N^{\frac{1}{2}}} \big) \big( x- \frac{t \theta}{\sqrt{\theta^2 + m^2}},t \big) dt \big|^2 dx &\lesssim N \int_{\R^2} |\mathcal{M} \tilde{g}|^2(x,y) dx dy \\
&\lesssim N \int_{\R^2} |\tilde{\mathcal{M}} \tilde{g}|^2(x,y) dx dy \\
&\lesssim N \log(N)^ 2 \| \tilde{g} \|_{L^2}^2.
\end{split}
\end{equation}
The first factor is estimated by a standard square function estimate (see e.g. \cite[Theorem~3]{Cordoba1979}):
\begin{equation}
\label{eq:SquareFunctionEstimate}
\big( \int \big| \sum_{\theta} |f_\theta(x)|^2 \big|^2 \big)^{\frac{1}{2}} \lesssim \| f \|_{L^4}^2.
\end{equation}

We finish the proof of \eqref{eq:LocalSmoothing1d} and hence Theorem \ref{thm:LocalSmoothing1d} by plugging \eqref{eq:MaximalFunctionConsequence} and \eqref{eq:SquareFunctionEstimate} into \eqref{eq:ReductionMaximalFunction}. \end{proof}


\section{Local smoothing in higher dimensions}
\label{section:ProofLocalSmoothingHigher}

For the proof of Theorem \ref{thm:GeneralInitialData}, we use decoupling inequalities which are sensitive to the degeneracy into the radial direction.

\subsection{Decoupling for radially degenerate elliptic surfaces}

In the following we show decoupling estimates for the surface $S = \{(\xi,\sqrt{|\xi|^2 + m^2}) : \, \xi \in \R^d, \; \frac{1}{2} \leq |\xi| \leq 2 \}$. Firstly we suppose that $\frac{1}{N} \leq m^2 \leq 1$. The wave regime $m^2 \lesssim \frac{1}{N}$ is treated separately in Section \ref{subsection:DegenerateDecoupling}. By the considerations in the proof of Proposition \ref{prop:NecessaryConditions}, $S$ has one principal curvature of size $\sim m^2$ in the radial direction and the remaining principal curvatures in angular directions are of size $\sim 1$. This dictates decoupling into rectangles of thickness $N^{-1}$, which have radial length $N^{-\frac{1}{2}} / m$ and in the angular directions size $N^{-\frac{1}{2}}$ because those frequency supports trivialize the Fourier extension operator $\mathcal{E}_{m^2}$. By $(\alpha,\beta)$-sectors we refer to sectors in the unit annulus of length $\alpha$ into the radial direction and length $\beta$ into angular direction. The sums in the decoupling inequalities below are over essentially disjoint sectors covering the unit annulus. For a sector $\theta$ we write
\begin{equation*}
\mathcal{E}_{m^2} f_{\theta} = \int_{\theta} e^{i(x' \cdot \xi + t \sqrt{|\xi|^2 + m^2})} f(\xi) \, d\xi.
\end{equation*}

\begin{proposition}
\label{prop:DecouplingRadiallyDegenerate}
Let $N \in 2^{\N_0}$ and $\frac{1}{N} \leq m^2 \leq 1$. Then the following estimate holds:
\begin{equation*}
\| \mathcal{E}_{m^2} f \|_{L^p_{t,x}(B_{d+1}(0,N))} \lesssim_\varepsilon N^\varepsilon \big( \sum_{\theta: (N^{-\frac{1}{2}}/m, N^{-\frac{1}{2}})-\text{sectors}} \| \mathcal{E}_{m^2} f_{\theta} \|_{L_{t,x}^p(w_{B_{d+1}(0,N)})}^2 \big)^{\frac{1}{2}}
\end{equation*}
provided that $2 \leq p \leq \frac{2(d+2)}{d}$.
\end{proposition}
For $m \sim 1$ this is evident from the $\ell^2$-decoupling for uniformly elliptic surfaces due to Bourgain--Demeter \cite[Section~7]{BourgainDemeter2015}.

\smallskip

For $\frac{1}{N} \leq m^2 \leq 1$, the proof is carried out in two steps: Firstly, we use a Pramanik--Seeger \cite{PramanikSeeger2007} argument to decouple sectors of radial length $N^{-\varepsilon}$ with aperture $N^{-\varepsilon}$ into sectors of radial length $N^{-\varepsilon}$ and aperture size $m$.
\begin{proposition}
\label{prop:DecouplingConeRadiallyDegenerate}
Let $2 \leq p \leq \frac{2(d+1)}{d-1}$ and suppose that $f: \R^d \to \C$ is supported in a sector of radial length $N^{-\varepsilon}$ and aperture $N^{-\varepsilon}$ contained in the unit annulus. Then the following decoupling inequality holds:
\begin{equation*}
\| \mathcal{E}_{m^2} f \|_{L^p_{t,x}(w_{B_{d+1}(0,N)})} \lesssim_\delta N^\delta \big( \sum_{\theta : (N^{-\varepsilon},m)-\text{sectors}} \| \mathcal{E}_{m^2} f_\theta \|^2_{L_{t,x}^p(w_{B_{d+1}(0,N)})} \big)^{\frac{1}{2}}.
\end{equation*}
\end{proposition}
\begin{proof}
This is a consequence of the decoupling inequality for the cone. Note the following:
\begin{equation*}
\sqrt{|\xi|^2 + m^2} - |\xi| = \frac{m^2}{\sqrt{|\xi|^2+ m^2} + |\xi|}.
\end{equation*}
This means the $\frac{1}{N}$-neighbourhood of $\{(\xi,\sqrt{|\xi|^2+m^2}) : \frac{1}{2} \leq |\xi| \leq 2\}$ is contained in the $m^2$-neighbourhood of $\{ (\xi,|\xi|) : \frac{1}{2} \leq |\xi| \leq 2 \}$. The decoupling inequality due to Bourgain--Demeter \cite[Theorem~1.2]{BourgainDemeter2015} for a function $F: \R^{d+1} \to \C$ with Fourier support in the $m^2$-neighbourhood of the cone is given by
\begin{equation*}
\begin{split}
&\quad \big\| \sum_{\theta: (N^{-\varepsilon},m) - \text{sector}} F_{\theta} \big\|_{L^{\frac{2(d+1)}{d-1}}(w_{B_{d+1}(0,N))})} \\
&\lesssim_\varepsilon m^\varepsilon \big( \sum_{\theta: (N^{-\varepsilon},m) - \text{sector}} \| F_{\theta} \|^2_{L^{\frac{2(d+1)}{d-1}}(w_{B_{d+1}(0,N)})} \big)^{\frac{1}{2}}.
\end{split}
\end{equation*}
The proof is concluded by letting $F_{\theta} = \mathcal{E}_{m^2} f_{\theta}$.
\end{proof}

So far we have not used the ellipticity of the surface into the radial direction. This explains why the decoupling exponent in the above proposition matches the one for the cone. We use anisotropic rescaling to take into account the partial degeneracy in the radial direction:
\begin{proposition}
\label{prop:SmallDecouplingRadiallyDegenerate}
Let $N \in 2^{\N_0}$, $N \gg 1$, $\frac{1}{N} \leq m^2 \leq 1$, $\varepsilon > 0$ and $\text{supp}(f)$ is contained in a $(N^{-\varepsilon},mN^{-\varepsilon})$-sector within the unit annulus. Then the following estimate holds:
\begin{equation*}
\| \mathcal{E}_{m^2} f \|_{L^p(B_{d+1}(0,N))} \lesssim_\varepsilon N^\varepsilon \big( \sum_{\theta: (N^{-\frac{1}{2}} /m , N^{-\frac{1}{2}})-\text{sector} } \| \mathcal{E}_{m^2} f_{\theta} \|^2_{L^p(w_{B_{d+1}(0,N)})} \big)^{\frac{1}{2}}
\end{equation*}
for $2 \leq p \leq \frac{2(d+2)}{d}$.
\end{proposition}
\begin{proof}
We suppose by spherical symmetry that the radial direction of the sector is $e_1$ and consider the rescaled function
\begin{equation*}
\psi(\xi_1,\xi') = \varphi(\xi_1/m, \xi'); \quad \varphi(\xi) = \sqrt{\xi_1^2 + |\xi'|^2 + m^2}.
\end{equation*}
We compute
\begin{equation*}
\partial^2_{ij} \varphi(\xi)= \frac{(|\xi|^2 + m^2) \delta_{ij} - \xi_i \xi_j}{(|\xi|^2 + m^2)^{\frac{3}{2}}}, \quad \xi_1 \in [a,a+N^{-\varepsilon}], \; \xi' \in B_{d-1}(0,mN^{-\varepsilon}).
\end{equation*}
Consequently,
\begin{equation*}
\partial^2_{11} \varphi = \frac{|\xi'|^2 + m^2}{(|\xi|^2 + m^2)^{\frac{3}{2}}} = O(m^2), \quad \partial^2_{ii} \varphi = \frac{|\xi_1|^2 + m^2 + |\xi'|^2 - \xi_i^2}{(|\xi|^2 + m^2)^{\frac{3}{2}}} = O(1), \quad (i \geq 2).
\end{equation*}
Moreover,
\begin{equation*}
\begin{split}
\partial^2_{ij} \varphi &= - \frac{\xi_i \xi_j}{(|\xi|^2 + m^2)^{\frac{3}{2}}} = O(m N^{- \varepsilon}) \text{ for } 1 = i < j \leq d, \\
\partial^2_{ij } \varphi &= - \frac{\xi_i \xi_j}{(|\xi|^2 + m^2)^{\frac{3}{2}}} = O(m^2 N^{-2 \varepsilon}) \text{ for } 2 \leq i < j \leq d.
\end{split}
\end{equation*}
Anisotropic rescaling shows that indeed for the surface $\{(\xi,\psi(\xi)) : \, \xi \in [ma,ma+mN^{-\varepsilon}] \times B_{d-1}(0,mN^{-\varepsilon}) \}$ we find all the principal curvatures to be comparable to $1$ because $\partial_{\xi \xi}^2 \psi = 1 + O(N^{-\varepsilon})$. 

This allows us to apply the Bourgain--Demeter \cite[Section~7]{BourgainDemeter2015} decoupling result for elliptic surfaces as follows: Write
\begin{equation}
\label{eq:AuxiliaryDecouplingElliptic}
\begin{split}
&\quad \| \mathcal{E}_{m^2} f \|_{L^p(B_{d+1}(0,N))} \\
&\leq \big\| w_{B_{d+1}(0,2N)} \int e^{i(x \cdot \xi + t \sqrt{|\xi|^2 + m^2})} f(\xi) d\xi \big\|_{L_{t,x}^p([-N,N] \times [-N,N] \times [-N,N]^{d-1})}. 
\end{split}
\end{equation}
We use the change of variables $\xi_1^* = m \xi_1$, $x_1' = m^{-1} x_1$ and write \\
$\tilde{f}(\xi^*) = m^{\frac{1}{p}-1} f(m \xi_1^*, (\xi')^*)$. We find
\begin{equation*}
\eqref{eq:AuxiliaryDecouplingElliptic} = \big\| \int e^{i( x' \cdot \xi^* + t \psi(\xi^*))} \tilde{f}(\xi^*) d \xi^* \big\|_{L_{t,x'}^p([-N,N] \times [-\frac{N}{m},\frac{N}{m}] \times [-N,N]^{d-1})}.
\end{equation*}
Let $E_{m,N} = [-N,N] \times [-\frac{N}{m},\frac{N}{m}] \times [-N,N]^{d-1}$. We cover $E_{m,N}$ with finitely overlapping balls of size $N$:
\begin{equation}
\label{eq:CoveringEllipsoid}
\big\| \int e^{i ( x' \cdot \xi^* + t \psi(\xi^*))} \tilde{f}(\xi^*) d\xi^* \big\|^p_{L_{t,x}^p(E_{m,N})} \lesssim \sum_{B_N} \big\| \int e^{i(x' \cdot \xi^* + t \psi(\xi^*))} \tilde{f}(\xi^*) d \xi^* \big\|^p_{L^p(B_N)}.
\end{equation}
By translation invariance, we can apply decoupling on every ball:
\begin{equation*}
\begin{split}
&\quad \big\| \int e^{i( x' \cdot \xi^* + t \psi(\xi^*))} \tilde{f}(\xi^*) d\xi^* \big\|_{L^p(B_N)} \\ &\lesssim_\varepsilon N^\varepsilon \big( \sum_{\theta: N^{-\frac{1}{2}}-\text{ball}} \big\| \int e^{i(x' \cdot \xi^* + t \psi(\xi^*))} \tilde{f}_\theta(\xi^*) d\xi^* \big\|^2_{L^p(w_{B_N})} \big)^{\frac{1}{2}}.
\end{split}
\end{equation*}
Plugging this into \eqref{eq:CoveringEllipsoid} and applying Minkowski's inequality yields
\begin{equation*}
\begin{split}
&\quad \big\| \int e^{i(x' \cdot \xi^* + t \psi(\xi^*))} \tilde{f}(\xi^*) d\xi^* \big\|_{L^p(E_{m,N})} \\
&\lesssim_\varepsilon N^\varepsilon \big( \sum_{\theta: N^{-\frac{1}{2}}-\text{ball}} \int e^{i(x' \cdot \xi^* + t \psi(\xi^*))} \tilde{f}(\xi^*) d\xi^* \big\|^2_{L^p(E_{m,N})} \big)^{\frac{1}{2}}.
\end{split}
\end{equation*}
Now we can reverse the anisotropic scaling to obtain the claim. We obtain a decoupling into rectangles of size $N^{-\frac{1}{2}}/m \times N^{-\frac{1}{2}}$ into directions $\xi_1$ and $\xi'$:
\begin{equation*}
\| \mathcal{E}_{m^2} f \|_{L^p(B_{d+1}(0,N))} \lesssim_\varepsilon N^\varepsilon \big( \sum_{r: (N^{-\frac{1}{2}} /m , N^{-\frac{1}{2}})-\text{rectangle} } \| \mathcal{E}_{m^2} f_{r} \|^2_{L^p(w_{B_{d+1}(0,N)})} \big)^{\frac{1}{2}}
\end{equation*}

We need to check that the rectangles are comparable to sectors of size $N^{-\frac{1}{2}}/m$ and $N^{-\frac{1}{2}}$ into radial and angular direction. This is indeed the case: the radial direction at the edge of the $(N^{-\varepsilon},mN^{-\varepsilon})$-sector is given by
\begin{equation*}
\mathfrak{n}_r = \frac{a e_1 + mN^{-\varepsilon} e'}{\sqrt{a^2+m^2 N^{-2 \varepsilon}}}, \quad e' \cdot e_1 = 0, \quad \| e' \| = 1.
\end{equation*}
Taking into account the length $N^{-\frac{1}{2}}/m$ into the radial direction we find the difference into angular direction to be of size
\begin{equation*}
| e' \cdot (\mathfrak{n}_r \cdot \frac{N^{-\frac{1}{2}}}{m}) | \lesssim N^{-\frac{1}{2}-\varepsilon}.
\end{equation*}
Since $N^{-\frac{1}{2}-\varepsilon} \ll N^{-\frac{1}{2}}$ this shows overlap into angular direction to be of $O(1)$. Angular directions are given by
\begin{equation*}
\mathfrak{n}_{\omega} = \frac{ mN^{-\varepsilon} e_1 + a e'' }{\sqrt{a^2+m^2 N^{-2 \varepsilon}}}, \quad e'' \cdot e_1 = 0, \quad \| e'' \| = 1.
\end{equation*}
Clearly, 
\begin{equation*}
|e'' \cdot( \mathfrak{n}_{\omega} \cdot N^{-\frac{1}{2}})| \lesssim N^{-\frac{1}{2}}.
\end{equation*}
This verifies comparability of the $N^{-\frac{1}{2}}/m \times N^{-\frac{1}{2}}$-rectangles and completes the proof.
\end{proof}

We conclude the proof of Proposition \ref{prop:DecouplingRadiallyDegenerate} by successive applications of Propositions \ref{prop:DecouplingConeRadiallyDegenerate} and \ref{prop:SmallDecouplingRadiallyDegenerate}. 
\begin{proof}[Proof~of~Proposition~
\ref{prop:DecouplingRadiallyDegenerate}]
We decompose by Minkowski's inequality and the Cauchy Schwarz inequality
\begin{equation*}
\| \mathcal{E}_{m^2} f \|_{L^p_{t,x}(B_{d+1}(0,N))} \lesssim_\varepsilon N^{\frac{d \varepsilon}{2}} \big( \sum_{\theta_1: N^{-\varepsilon}-\text{ball}} \| \mathcal{E}_{m^2} f_{\theta_1} \|^2_{L^p_{t,x}(w_{B_{d+1}(0,N)})} \big)^{\frac{1}{2}}.
\end{equation*}
Now we can use Proposition \ref{prop:DecouplingConeRadiallyDegenerate} to find
\begin{equation*}
\| \mathcal{E}_{m^2} f_{\theta_1} \|_{L^p_{t,x}} \lesssim_\delta N^{\delta} \big( \sum_{\substack{\theta_2: (N^{-\varepsilon},m)-\text{sector} \\ \text{covering } \theta_1}} \| \mathcal{E}_{m^2} f_{\theta_2} \|_{L^p_{t,x}(w_{B_{d+1}}(0,N))}^2 \big)^{\frac{1}{2}}.
\end{equation*}
We use another trivial decoupling to further decompose the sectors in angular variables from size $m$ to size $mN^{-\varepsilon}$ which incurs another factor of $N^{\frac{d \varepsilon}{2}}$. Collecting the previuos estimates we have
\begin{equation}
\label{eq:AuxDecouplingI}
\| \mathcal{E}_{m^2} f \|_{L^p_{t,x}(B_{d+1}(0,N))} \lesssim_\varepsilon N^{(d+1) \varepsilon} \big( \sum_{\theta_3: (N^{-\varepsilon},m N^{-\varepsilon})-\text{sector}} \| \mathcal{E}_{m^2} f_{\theta_3} \|_{L^p_{t,x}(w_{B_{d+1}}(0,N))}^2 \big)^{\frac{1}{2}}.
\end{equation}
At this point we can invoke Proposition \ref{prop:SmallDecouplingRadiallyDegenerate} to find for $\theta_3$ like in the previous display:
\begin{equation}
\label{eq:AuxDecouplingII}
\| \mathcal{E}_{m^2} f_{\theta_3} \|_{L^p_{t,x}(w_{B_{d+1}}(0,N))} \lesssim_\varepsilon N^{\varepsilon} \big( \sum_{\substack{\theta_4: (N^{-\frac{1}{2}}/m,N^{-\frac{1}{2}})-\text{sector} \\ \text{ covering } \theta_3}} \| \mathcal{E}_{m^2} f_{\theta_4} \|_{L^p_{t,x}(w_{B_{d+1}(0,N)})}^2 \big)^{\frac{1}{2}}.
\end{equation}
Taking \eqref{eq:AuxDecouplingI} and \eqref{eq:AuxDecouplingII} together we find
\begin{equation*}
\| \mathcal{E}_{m^2} f \|_{L^p_{t,x}(B_{d+1}(0,N))} \lesssim_\varepsilon N^{C \varepsilon} ( \sum_{\theta: (N^{-\frac{1}{2}}/m,N^{-\frac{1}{2}})-\text{sector}} \| \mathcal{E}_{m^2} f_{\theta} \|^2_{L^p_{t,x}(w_{B_{d+1}(0,N)})} \big)^{\frac{1}{2}}.
\end{equation*}
Letting $\varepsilon \to \varepsilon / C$ the proof is complete.
\end{proof}

\subsection{Decoupling for uniformly degenerate surfaces}

We turn to the easier case of $1 \ll m^2 \lesssim N^2$. Recall that in this case the surface $\{(\xi, \sqrt{|\xi|^2 + m^2}) : \frac{1}{2} \leq |\xi| \leq 2 \}$ is degenerate as well in the radial as angular direction to the same extent. We show the following:
\begin{proposition}
\label{prop:UniformlyDegenerateDecoupling}
Let $2 \leq p \leq \frac{2(d+1)}{d-1}$ and $1 \ll m^2 \lesssim N$. Then the following estimate holds:
\begin{equation*}
\big\| \mathcal{E}_{m^2} f \big\|_{L^p_{t,x}(w_{B_{N}})} \lesssim_\varepsilon N^\varepsilon \big( \sum_{\theta: m^{\frac{1}{2}} N^{-\frac{1}{2}}-\text{ball}} \| \mathcal{E}_{m^2} f_{\theta} \|^2_{L^p_{t,x}(w_{B_N})} \big)^{\frac{1}{2}}.
\end{equation*}
\end{proposition}
\begin{proof}
This follows from a change of variables: Let $t' = mt$, $x'=mx$, $\xi' = \xi/m$ to find
\begin{equation*}
\begin{split}
&\quad \big\| \int_{\{ \frac{1}{2} \leq |\xi| \leq 2\}} e^{i(x \cdot \xi + t \sqrt{|\xi|^2+ m^2})} f(\xi) d\xi \big\|_{L^p(w_{B_{d+1}}(0,N))} \\
&= \big\| \int_{ \{ \frac{1}{2m} \leq |\xi'| \leq \frac{2}{m} \} } e^{i(x' \cdot \xi' + t' \sqrt{|\xi'|^2+ 1})} \tilde{f}(\xi') d \xi' \big\|_{L^p_{t',x'}(w_{B_{d+1}(0,mN)})}.
\end{split}
\end{equation*}
Since $\{(\xi', \sqrt{|\xi'|^2+1}) : |\xi'| \sim \frac{1}{m} \}$ is a uniformly elliptic surface, the Bourgain--Demeter $\ell^2$-decoupling theorem is applicable at scale $mN$. This yields
\begin{equation*}
\begin{split}
&\quad \big\| \int_{ \{ \frac{1}{2m} \leq |\xi'| \leq \frac{2}{m} \}} e^{i(x' \cdot \xi' + t' \sqrt{|\xi'|^2 + 1})} \tilde{f}(\xi') d\xi' \big\|_{L^p_{t,x}(w_{B_{d+1}(0,Nm)})} \\ &\lesssim_\varepsilon N^\varepsilon \big( \sum_{\theta: (Nm)^{-\frac{1}{2}}-\text{ball}} \| \mathcal{E}_1 \tilde{f}_\theta \|^2_{L^p_{t',x'}(w_{B_{d+1}(0,Nm)})} \big)^{\frac{1}{2}}.
\end{split}
\end{equation*}
Now we reverse the change of variables to find
\begin{equation*}
\big( \sum_{\theta: (Nm)^{-\frac{1}{2}}-\text{ball}} \| \mathcal{E}_1 \tilde{f}_\theta \|^2_{L^p_{t',x'}(w_{B_{d+1}(0,Nm)})} \lesssim \big( \sum_{\theta: m^{\frac{1}{2}} N^{-\frac{1}{2}}-\text{ball}} \| \mathcal{E}_{m^2} f_\theta \big\|^2_{L^p_{t,x}(w_{B_{d+1}(0,N))}} \big)^{\frac{1}{2}}.
\end{equation*}
Taking the estimates together we complete the proof.
\end{proof}

\subsection{Degenerate cases}
\label{subsection:DegenerateDecoupling}

Finally, we record the degenerate cases $m^2 \lesssim \frac{1}{N}$, in which case the characteristic surface is indistinguishable from the cone, and $m^2 \gtrsim N$, in which case the characteristic surface is essentially flat. In the latter case there is no dispersion for $S_{m^2}$ anymore, and we have the following:
\begin{proposition}
Let $2 \leq p \leq \infty$, $N \in 2^{\N_0}$, $m^2 \gtrsim N$. Then the following estimate holds:
\begin{equation*}
\| S_{m^2} f \|_{L^p_{t,x}(B_{d+1}(0,N))} \lesssim N^{\frac{1}{p}} \| f \|_{L^p(\R^d)}.
\end{equation*}
\end{proposition}

For $m^2 \lesssim \frac{1}{N}$, the characteristic surface is in the $1/N$-neighbourhood of the cone $\{(\xi,|\xi|) : \xi \in \R^d, \; \frac{1}{2} \leq |\xi| \leq 2 \}$ and we have the following cone decoupling estimate:
\begin{proposition}
\label{prop:ApproxConeDecoupling}
Let $N \in 2^{\N_0}$ and $m^2 \lesssim \frac{1}{N}$. Then the following estimate holds:
\begin{equation*}
\| \mathcal{E}_{m^2} f \|_{L^p_{t,x}(B_{d+1}(0,N))}  \lesssim_\varepsilon N^\varepsilon \big( \sum_{\theta: N^{-\frac{1}{2}}-\text{sector}} \| \mathcal{E}_{m^2} f_{\theta} \|_{L^p_{t,x}(w_{B_{d+1}(0,N)})}^2 \big)^{\frac{1}{2}}
\end{equation*}
provided that $2 \leq p \leq \frac{2(d+1)}{d-1}$.
\end{proposition}
\begin{proof}
This follows from from observing that
\begin{equation*}
\sqrt{|\xi|^2 + m^2} - |\xi| = \frac{m^2}{\sqrt{|\xi|^2 + m^2} + |\xi|}.
\end{equation*}
Since $m^2 \lesssim \frac{1}{N}$, the characteristic surface is in the $1/N$-neighbourhood of the cone $\{(\xi,|\xi|) : \xi \in \R^d, \; \frac{1}{2} \leq |\xi| \leq 2 \}$,  and the claim is immediate from \cite[Theorem~1.2]{BourgainDemeter2015}.
\end{proof}


\subsection{Proofs of Theorems \ref{thm:CompactlySupportedInitialData} and \ref{thm:GeneralInitialData}}

With the suitable decoupling estimates at hand, the argument is standard and we shall be brief. 
We start with the proof of Theorem \ref{thm:CompactlySupportedInitialData}:
\begin{proof}[Proof~of~Theorem~\ref{thm:CompactlySupportedInitialData}]
We use the reductions from Sections \ref{section:Preliminaries} and \ref{section:LinearizationReduction}. Since $u_0$ is compactly supported, we can choose $N_0$ large enough such that $\text{supp}(u_0) \subseteq B(0,N_0)$. For $N \leq N_0$ we can use pointwise estimates to find
\begin{equation*}
\| \cos(t \sqrt{\mathcal{H}}) P_{\lesssim N_0} f \|_{L^p_{t,x}([0,1] \times \R^d)} \lesssim N_0^C \| f \|_{L^p_x}.
\end{equation*}
This is based on the fact that the simultaneous localization in phase space $\{|x| \lesssim N_0, \; |\xi| \lesssim N_0 \}$ corresponds to the contribution of the spectral projection $\mathfrak{P}_{\lesssim N_0} u_0$.

So, we can focus on $N \geq N_0$, for which we suppose that the parametrix exists on $[0,1]$. By the reductions from Section \ref{section:LinearizationReduction} (see Proposition \ref{prop:KleinGordonReduction}) it suffices to prove the estimate
\begin{equation}
\label{eq:AuxLocalSmoothingWaveRegime}
\| S_{m^2} f \|_{L^p_{t,x}(B_{d+1}(0,N))} \lesssim N^{s+\frac{1}{p}} \| f \|_{L^p(\R^d)}
\end{equation}
to obtain
\begin{equation*}
\| \cos(t \sqrt{\mathcal{H}}) P_N f \|_{L^p_{t,x}([0,1] \times \R^d)} \lesssim N^{s} \| f \|_{L^p(\R^d)}.
\end{equation*}
In the present context we have $m^2 \lesssim \frac{1}{N}$ by the support conditions on $f$ and choosing $N$ large enough.

\medskip

To prove \eqref{eq:AuxLocalSmoothingWaveRegime} for $m^2 \lesssim \frac{1}{N}$ we use Proposition \ref{prop:ApproxConeDecoupling}: Write
\begin{equation*}
S_{m^2} f_{\theta} = \int_{\theta} e^{i(x' \cdot \xi + t \sqrt{|\xi|^2 + m^2})} \hat{f}(\xi) d\xi.
\end{equation*}
We find for $2 \leq p \leq \frac{2(d+1)}{d-1}$
\begin{equation*}
\| S_{m^2} f \|_{L^p_{t,x}(B_{d+1}(0,N))} \lesssim_\varepsilon N^\varepsilon \big( \sum_{\theta: N^{-\frac{1}{2}}-\text{sector}} \| S_{m^2} f_{\theta} \|_{L^p}^2 \big)^{\frac{1}{2}}.
\end{equation*}
It is straight-forward that
\begin{equation*}
S_{m^2} f_{\theta}(x,t) = \int K_\theta(x,t;y) f(y) dy
\end{equation*}
with
\begin{equation*}
\sup_{x \in \R^d} \int_{\R^d} |K_\theta(x,t;y)| dy + \sup_{y \in \R^d} \int_{\R^d} |K_\theta(x,t;y)| dx \leq C
\end{equation*}
for $|t| \lesssim N$. Hence, we obtain from a fixed-time kernel estimate
\begin{equation*}
\big( \sum_{\theta: N^{-\frac{1}{2}}-\text{sector}} \| S_{m^2} f_{\theta} \|^2_{L^p(w_{B_{d+1}}(0,N))} \big)^{\frac{1}{2}} \lesssim N^{\frac{1}{p}} \big( \sum_{\theta: N^{-\frac{1}{2}}-\text{sector}} \| P_\theta f \|^2_{L^p(\R^d)} \big)^{\frac{1}{2}}.
\end{equation*}
Recall the standard kernel estimate for $\theta$ an $N^{-\frac{1}{2}}$-sector:
\begin{equation*}
\| P_{\theta} f \|_{L^p(\R^d)} \lesssim \| f \|_{L^p(\R^d)}.
\end{equation*}
This implies by interpolation between $p=2$ and $p=\infty$:
\begin{equation}
\label{eq:lpLpSummation}
\big( \sum_{\theta:N^{-\frac{1}{2}}} \| P_{\theta} f \|^p_{L^p(\R^d)} \big)^{\frac{1}{p}} \lesssim \| f \|_{L^p(\R^d)}.
\end{equation}

Now we can use H\"older's inequality to change summation from $\ell^2$ to $\ell^p$ and \eqref{eq:lpLpSummation} to obtain
\begin{equation*}
\begin{split}
\big( \sum_{\theta: (1,N^{-\frac{1}{2}})-\text{sector}} \| P_\theta f \|^2_{L^p(\R^d)} \big)^{\frac{1}{2}} &\lesssim N^{\frac{d-1}{2} \big( \frac{1}{2}- \frac{1}{p} \big)} \big( \sum_{\theta: N^{-\frac{1}{2}}-\text{sector}} \| P_{\theta} f \|^p_{L^p(\R^d)} \big)^{\frac{1}{p}} \\
&\lesssim N^{\frac{d-1}{2} \big( \frac{1}{2} - \frac{1}{p} \big)} \| f \|_{L^p(\R^d)}.
\end{split}
\end{equation*}
This shows the local smoothing estimate claimed in Theorem \ref{thm:CompactlySupportedInitialData} for $p=\frac{2(d+1)}{d-1}$.

 The claim for $p \geq \frac{2(d+1)}{d-1}$ follows from pointwise estimates for large $q$ provided by Proposition \ref{prop:PointwiseEstimates} and interpolation. This is similar to the argument at the end of the proof of Theorem \ref{thm:LocalSmoothing1d}. The proof is complete.
\end{proof}

Now we turn to the proof of Theorem \ref{thm:GeneralInitialData}:

\begin{proof}[Proof~of~Theorem~\ref{thm:GeneralInitialData}]
We use Proposition \ref{prop:KleinGordonReduction} to reduce to Klein-Gordon smoothing estimates
\begin{equation}
\label{eq:KleinGordonSmoothingGeneralData}
\| S_{m^2} f \|_{L^p_{t,x}(B_{d+1}(0,N))} \lesssim N^s \| f \|_{L^p(\R^d)}.
\end{equation}
But in the present context we can no longer suppose that $m^2 \lesssim \frac{1}{N}$. However, for $m^2 \lesssim \frac{1}{N}$, the estimates from the proof of Theorem \ref{thm:CompactlySupportedInitialData} actually show  estimates with less derivative loss than presently claimed. The larger derivative loss stems from the elliptic regime. For $\frac{1}{N} \lesssim m^2 \lesssim 1$ we can use the decoupling estimates from Proposition \ref{prop:DecouplingRadiallyDegenerate} to find for $2 \leq p \leq \frac{2(d+2)}{d}$:
\begin{equation*}
\| S_{m^2} f \|_{L^p_{t,x}(B_{d+1}(0,N))} \lesssim_\varepsilon N^\varepsilon \big( \sum_{\theta: (N^{-\frac{1}{2}}/m,N^{-\frac{1}{2}})-\text{sector}} \| S_{m^2} f_{\theta} \|^2_{L^p_{t,x}(w_{B_{d+1}(0,N))})} \big)^{\frac{1}{2}}.
\end{equation*}
This strategy is the same like in the proof of Theorem \ref{thm:GeneralInitialData}: we use a kernel estimate to estimate the propagation at fixed times, which incurs a factor of $N^{\frac{1}{p}}$ by additional integration in $t$, then we use H\"older's inequality to change from $\ell^2$ to $\ell^p$. Finally we can use a kernel estimate for frequency projections like in \eqref{eq:lpLpSummation} to conclude the argument.

\medskip

The only estimate which changes compared to the above is H\"older's inequality to increase from $\ell^2$ to $\ell^p$: We have
\begin{equation*}
\begin{split}
&\quad \big( \sum_{\theta: (N^{-\frac{1}{2}}/m,N^{-\frac{1}{2}})-\text{sector}} \| S_{m^2} f_{\theta} \|^2_{L^p} \big)^{\frac{1}{2}} \\
 &\lesssim N^{\frac{1}{p}} \big( \sum_{\theta:(N^{-\frac{1}{2}}/m,N^{-\frac{1}{2}})-\text{sector}} \| P_{\theta} f \|^2_{L^p(\R^d)} \big)^{\frac{1}{2}} \\
&\lesssim N^{\frac{1}{p}} N^{\frac{d-1}{2} \big( \frac{1}{2} - \frac{1}{p} \big)} (mN^{\frac{1}{2}})^{\frac{1}{2}-\frac{1}{p}} \big( \sum_{\theta: (N^{-\frac{1}{2}}/m,N^{-\frac{1}{2}})-\text{sector}} \| P_\theta f \|_{L^p}^p \big)^{\frac{1}{p}} \\
&\lesssim N^{\frac{1}{p}+\frac{d-1}{2} \big( \frac{1}{2} - \frac{1}{p} \big)} (mN^{\frac{1}{2}})^{\frac{1}{2}-\frac{1}{p}} \| f \|_{L^p}.
\end{split}
\end{equation*}
The constant becomes largest for $m \sim 1$ in which case we find
\eqref{eq:KleinGordonSmoothingGeneralData} with $s = \frac{1}{p} + \frac{d}{2} \big( \frac{1}{2} - \frac{1}{p} \big)$. 

\smallskip

It remains to check the case $m \gtrsim 1$: An application of Proposition \ref{prop:UniformlyDegenerateDecoupling} yields
\begin{equation*}
\| S_{m^2} f \|_{L^p_{t,x}(B_{d+1}(0,N))} \lesssim_\varepsilon N^\varepsilon \big( \sum_{\theta: (m^{\frac{1}{2}} N^{-\frac{1}{2}})-\text{ball}} \| S_{m^2} f_{\theta} \|_{L^p_{t,x}(w_{B_{d+1}(0,N)})}^2 \big)^{\frac{1}{2}}.
\end{equation*}
Recall that for $m^2 \gtrsim N$ the evolution of $S_{m^2}$ is already trivial on the time-scale $N$ and the above holds without $N^\varepsilon$-loss. So it remains the case $1 \lesssim m^2 \lesssim N$. We can apply a kernel estimate followed by H\"older's inequality and the summation property of Fourier projections \eqref{eq:lpLpSummation} to find
\begin{equation*}
\begin{split}
&\quad \big( \sum_{\theta: (m^{\frac{1}{2}} N^{-\frac{1}{2}})-\text{ball}} \| S_{m^2} f_{\theta} \|_{L^p_{t,x}(w_{B_{d+1}(0,N)})}^2 \big)^{\frac{1}{2}} \\
 &\lesssim N^{\frac{1}{p}} \big( \sum_{\theta: (m^{\frac{1}{2}} N^{-\frac{1}{2}})-\text{ball}} \| P_{\theta} f \|_{L^p_{x}(\R^d)}^2 \big)^{\frac{1}{2}} \\
&\lesssim N^{\frac{1}{p}} (m^{-\frac{1}{2}} N^{\frac{1}{2}})^{d \big( \frac{1}{2} - \frac{1}{p} \big) } \big( \sum_{\theta: (m^{\frac{1}{2}} N^{-\frac{1}{2}})-\text{ball}} \| P_{\theta} f \|_{L^p_{x}(\R^d)}^p \big)^{\frac{1}{p}} \\
&\lesssim N^{\frac{1}{p}} (m^{-\frac{1}{2}} N^{\frac{1}{2}})^{d \big( \frac{1}{2} - \frac{1}{p} \big) } \| f \|_{L^p(\R^d)}.
\end{split}
\end{equation*}
Again the constant is maximized in case $m \sim 1$. In conclusion we have proved the estimate for $p=\frac{2(d+2)}{d}$:
\begin{equation*}
\| S_{m^2} f \|_{L^p_{t,x}(B_{d+1}(0,N))} \lesssim_\varepsilon N^\varepsilon N^{\frac{1}{p}} N^{\frac{d}{2} \big( \frac{1}{2} - \frac{1}{p} \big)} \| f \|_{L^p(\R^d)},
\end{equation*}
which is uniform in $m \in 2^{\Z}$.

\medskip

For $p= \frac{2(d+2)}{d}$ by Proposition \ref{prop:KleinGordonReduction} we obtain the smoothing estimate
\begin{equation*}
\| \cos(t \sqrt{\mathcal{H}}) f \|_{L^p_{t,x}([0,1] \times \R^d)} \lesssim \| f \|_{L^p_{s}(\R^d)}
\end{equation*}
for $s > d \big( \frac{1}{2} - \frac{1}{p} \big) - \frac{1}{p}$. To extend this for $p > \frac{2(d+2)}{d}$ we again interpolate with pointwise estimates for large $q$ provided by Proposition \ref{prop:PointwiseEstimates}. This completes the proof.
\end{proof}

\section{Implications for Bochner--Riesz estimates}
\label{section:BochnerRiesz}
\subsection{Global Bochner--Riesz means for the Hermite operator}

Next, we turn to the implications for the Hermite Bochner--Riesz means. By spectral calculus we define for $\alpha \geq 0$ and $x_+ = \max(x,0)$ the operator
\begin{equation}
\label{eq:BochnerRieszHermite}
\mathcal{B}^\alpha_{\lambda}(\mathcal{H}) f(x) = \big( 1 - \frac{\mathcal{H}}{\lambda} \big)^\alpha_+ f(x).
\end{equation}

In the following we say $\| \mathcal{B}^\alpha_\lambda(\mathcal{H}) \|_{L^p \to L^p}$ is uniformly bounded when it is bounded uniform in $\lambda \geq 1$.
In one dimension Askey--Wainger \cite{AskeyWainger1965} showed $L^p$-boundedness of $\mathcal{B}^\alpha_{\lambda}(\mathcal{H})$ for $\alpha = 0$ provided that $p \in (4/3,4)$. Moreover, Thangavelu \cite{Thangavelu1989} showed uniform boundedness for $p \in [1,\infty]$ provided that $\alpha > 1/6$. Hence, for $d=1$ it holds uniform boundedness provided that
\begin{equation*}
\alpha > \tilde{\gamma}(1,p) = \max \big( \frac{2}{3} \big| \frac{1}{2} - \frac{1}{p} \big| - \frac{1}{6}, 0 \big).
\end{equation*}
Thangavelu moreover showed necessity of $\alpha \geq \tilde{\gamma}(1,p)$. This settles the question in one dimension up to endpoints.

By a transplantation result due to Kenig--Stanton--Tomas \cite[Theorem~3]{KenigStantonTomas1982} (see also preceding work by Mitjagin \cite{Mitjagin1974}) uniform boundedness of the Hermite Bochner--Riesz means imply uniform boundedness for the Euclidean Bochner--Riesz means:
\begin{equation}
\label{eq:BochnerRieszEuclidean}
\mathcal{B}_{\lambda}^\alpha (\Delta) f(x) = \big( 1 + \frac{\Delta}{\lambda} \big)^{\alpha}_+ f(x).
\end{equation}
Recall that $\mathcal{B}^\alpha_\lambda(\Delta)$ is bounded on $L^p(\R)$ for any $\alpha \geq 0$ and $p \in (1,\infty)$  uniformly in $\lambda \geq 1$ due to boundedness of the Hilbert transform. The Bochner--Riesz conjecture states that \eqref{eq:BochnerRieszEuclidean} is bounded for $p \in [1,\infty] \backslash \{ 2 \}$ and $d \geq 2$ if and only if
\begin{equation*}
\alpha > \delta(d,p) := \max \big( d \big| \frac{1}{2} - \frac{1}{p} \big| - \frac{1}{2}, 0 \big).
\end{equation*}
Recall that the Bochner--Riesz conjecture implies the restriction conjecture (cf. \cite{Tao1999}) and $\alpha > 0$ is necessary due to a result by Fefferman \cite{Fefferman1971}. In two dimensions the Bochner--Riesz conjecture is settled (see  \cite{CarlesonSjoelin1972,Fefferman1973,
Hoermander1973,Carbery1983} and \cite{Stein1993} for further references). It remains open in dimensions $d \geq 3$. The most recent progress is reported by Guo--Wang--Zhang \cite{GuoWangZhang2022}. In the following we denote the operator norm of an operator $A: L^p(\R^d) \to L^p(\R^d)$ by $\| A \|_p$.

\medskip

Based on the transplantation result in higher dimensions it appeared conceivable that the boundedness properties of \eqref{eq:BochnerRieszHermite} and \eqref{eq:BochnerRieszEuclidean} coincide.
However, recently, Lee--Ryu \cite{LeeRyu2022} disproved this conjecture for the global Bochner--Riesz means \eqref{eq:BochnerRieszHermite}. They showed the following \cite[Proposition~4.2]{LeeRyu2022}:
\begin{proposition}[Necessary~conditions~for~Hermite~Bochner--Riesz~means]
Let $d \geq 1$ and $2 < p \leq \infty$. The uniform bound
$\| \mathcal{B}_\lambda^{\alpha}(\mathcal{H}) \|_p \leq C$ holds only if $\alpha > \delta(d,p)$ and
\begin{equation*}
\alpha \geq \gamma(d,p) = - \frac{1}{3p} + \frac{d}{3} \big( \frac{1}{2} - \frac{1}{p} \big).
\end{equation*}
\end{proposition}

This distinguishes $p_c = \frac{2(d+1)}{d}$ as critical exponent, for which $\| \mathcal{B}_\lambda^{\alpha}(\mathcal{H}) \|_{p} \leq C$ can be expected to hold uniform in $\lambda \geq 1$ for any $\alpha > 0$. Recall that the local smoothing conjecture for the Euclidean wave equation implies the Bochner--Riesz conjecture; see, e.g., \cite[Section~3.2]{BeltranHickmanSogge2021}. By following the argument we shall see now how the critical index $p_c$ for the Hermite Bochner--Riesz means relates to the critical index for local smoothing estimates for the wave Hermite equation. 
\begin{proposition}
Suppose that for $p \geq 2$, $\Lambda \gg 1$ the estimate
\begin{equation*}
\| e^{it \tau \sqrt{\mathcal{H}}} f \|_{L^p_{t,x}([-1,1] \times \R^d)} \lesssim_\varepsilon (\lambda \tau)^\varepsilon \| f \|_{L^p_x(\R^d)}
\end{equation*}
holds true for $\mathfrak{P}_\lambda f = f$. Then the estimate $\| \mathcal{B}^\alpha_\lambda(\mathcal{H}) \|_p \leq C$ holds for any $\alpha>0$ uniform in $\lambda \geq 1$.
\end{proposition}
\begin{proof}
We decompose
\begin{equation*}
\begin{split}
\big( 1 - \frac{\mathcal{H}}{\lambda} \big)^\delta_+ &= \lambda^{-\delta} (\sqrt{\lambda} - \sqrt{\mathcal{H}})^\delta_+ (\sqrt{\lambda} + \sqrt{\mathcal{H}})^\delta \chi_\lambda(\mathcal{H}) \\
&= \lambda^{-\delta} \big( \sum_{k=k_0}^{\infty} 2^{-\delta k } \chi(2^k( \sqrt{\lambda} - \sqrt{\mathcal{H}})) + \chi_0(\sqrt{\lambda} - \sqrt{\mathcal{H}}) \big) (\sqrt{\lambda} + \sqrt{\mathcal{H}})^\delta \chi_\lambda(\mathcal{H}).
\end{split}
\end{equation*}
Above $\chi$ denotes a suitable bump function, $\chi_0$ localizes to the smooth contribution $\mathcal{H} \ll \lambda$, and $\chi_\lambda$ denotes a smooth bump function adapted to $B(0,4\lambda)$. The ``smooth'' part is readily estimated:
\begin{equation*}
\| \lambda^{-\delta} \chi_0(\sqrt{\lambda} - \sqrt{\mathcal{H}}) (\sqrt{\lambda} + \sqrt{\mathcal{H}})^\delta \chi_{\lambda}(\mathcal{H}) \|_{p} \leq C.
\end{equation*}
For the main contribution from the singularity we use Fourier inversion:
\begin{equation*}
\chi(2^k(\sqrt{\lambda} - \sqrt{\mathcal{H}}) = \int e^{i 2^k t (\sqrt{\lambda} - \sqrt{\mathcal{H}})} \hat{\chi}(t) dt.
\end{equation*}
Since we can estimate 
\begin{equation*}
\| (\sqrt{\lambda} + \sqrt{\mathcal{H}})^\delta \chi_\lambda(\mathcal{H}) \|_{p} \lesssim \lambda^{\delta/2},
\end{equation*}
it remains to analyze
\begin{equation}
\label{eq:DecompositionBochnerRiesz}
\lambda^{-\frac{\delta}{2}} \big\| \sum_{k = k_0}^\infty 2^{-k \delta} \int_{\R} e^{i t 2^k (\sqrt{\lambda} - \sqrt{\mathcal{H}})} \hat{\chi}(t) f dt \|_{L^p(\R^d)}.
\end{equation}
By an application of Minkowski's inequality, H\"older's inequality, and the decay of $\hat{\chi}$, we find that it suffices to show
\begin{equation*}
\big\| e^{i 2^k t \sqrt{\mathcal{H}}} f \big\|_{L^p_{t,x}([-\lambda^{\delta/10},\lambda^{\delta/10}] \times \R^d)} \lesssim_\varepsilon 2^{k \varepsilon} \lambda^{\frac{\delta}{10}} \| f \|_{L^p}.
\end{equation*}
We make a change of variables and see by hypothesis
\begin{equation*}
\begin{split}
\big\| e^{i 2^k \sqrt{\mathcal{H}} t} f \big\|_{L^p_{t,x}([-\lambda^{\delta/10},\lambda^{\delta/10}] \times \R^d)} &= \lambda^{\frac{\delta}{10 p}} \| e^{i 2^k \lambda^{\frac{\delta}{10}} \sqrt{\mathcal{H}} t} f \|_{L^p_{t,x}([-1,1] \times \R^d)} \\
&\lesssim_\varepsilon 2^{k \varepsilon} \lambda^{\frac{\delta \varepsilon}{10} + \varepsilon} \lambda^{\frac{\delta}{10 p}} \| f \|_{L^p(\R^d)}.
\end{split}
\end{equation*}
Choosing $\varepsilon$ small enough and plugging this into \eqref{eq:DecompositionBochnerRiesz} we complete the proof.

\end{proof}

We remark that in one dimension, by some changes of variables the sharp local smoothing estimates established in Section \ref{section:ProofLocalSmoothing1d} indicate the uniform estimates
\begin{equation*}
\| \mathcal{B}_{\lambda}^{\alpha}(\mathcal{H}) \|_4 \leq C
\end{equation*}
for any $\alpha > 0$. This is the endpoint for global Hermite Bochner--Riesz summability in one dimension.

\subsection{Remarks on local Bochner--Riesz estimates}

Thangavelu \cite{Thangavelu1998} moreover considered local Bochner--Riesz estimates for the Hermite operator, i.e., $L^p$-bounds uniform in $\lambda \geq 1$
\begin{equation*}
\| \chi_E \mathcal{B}^\alpha_{\lambda}(\mathcal{H}) \chi_F \|_{p} \leq C
\end{equation*}
where $E$, $F \subseteq \R^d$ denote subsets of $\R^d$ and $\chi_E$ denotes the indicator function. Invoking for $E =F = B(0,\varepsilon)$ the transplantation result \cite[Theorem~3]{KenigStantonTomas1982} yields 
boundedness of the Euclidean Bochner--Riesz means. Recently, significant progress on local Hermite Bochner--Riesz means was made by Lee--Ryu \cite{LeeRyu2022}. So, one might still believe that the summability for the local Hermite Bochner--Riesz means coincides with the summability of the Euclidean Bochner--Riesz means. In this section we shall see that this likely not the case.

\medskip

 We recall notations from \cite{LeeRyu2022}. Let $d \geq 2$, and
\begin{equation*}
p_0(d) = 
\begin{cases}
2 \cdot \frac{3d+2}{3d-2} \text{ if } d \text{ is even}, \\
2 \cdot \frac{3d+1}{3d-3} \text{ if } d \text{ is odd},
\end{cases}
\end{equation*}
and
\begin{equation}
\label{eq:HermitianDistance}
\begin{split}
\mathcal{D}(x,y) &= 1 + (x \cdot y)^2 - |x|^2 - |y|^2, \quad (x,y) \in \R^d \times \R^d, \\
\mathfrak{D}(c_0) &= \{(x,y) \in \R^d \times \R^d : |x|,|y| \leq 1 - c_0, \quad \mathcal{D}(x,y) > c_0^2 \}.
\end{split}
\end{equation}

Lee--Ryu \cite{LeeRyu2022} found a condition on $E$, $F$ for local estimates to hold, which are not effected by their new counterexamples for the global estimates. They showed the following:
\begin{theorem}[{\cite[Theorem~1.2]{LeeRyu2022}}]
Let $d \geq 2$, and suppose that $E,F \subseteq \R^d$ are compact sets such that $ E \times F \subseteq \mathfrak{D}(c_0)$ for some $0<c_0<1$. Then there is a constant $C$ independent of $\lambda$ such that
\begin{equation*}
\| \chi_{E_\lambda} \mathcal{B}_\lambda^\alpha(\mathcal{H}) \chi_{F_\lambda} \|_{p \to p} \leq C,
\end{equation*}
provided that $p > p_0(d)$ and $\alpha > \alpha(d,p)$, where $E_\lambda$, $F_\lambda$ denote the dilated sets $\sqrt{\lambda} E$, $\sqrt{\lambda} F $, respectively.
\end{theorem}

The proof revolves around reducing to oscillatory integral estimates with phase functions satisfying ellipticity conditions. Then invoking the results of Guth--Hickman--Iliopoulou \cite{GuthHickmanIliopoulou2019} the above theorem follows.
In the following we point out that the phase function $\phi \in C^\infty(\R^n,\R^{n-1})$ in the oscillatory integral obtained by Lee--Ryu \cite{LeeRyu2022}
\begin{equation*}
T^\lambda f(x) = \int e^{i \phi^\lambda(x;\xi)} a(x;\xi) f(\xi) d\xi, \quad (x,\xi) = (x',x_n,\xi) \in \R^{d-1} \times \R \times \R^{d-1}
\end{equation*}
does in general not satisfy the Bourgain condition \cite{Bourgain1991,GuoWangZhang2022}. 
Firstly, we recall the more basic non-degeneracy and ellipticity assumptions on the phase function:
\begin{equation*}
\begin{split}
&H1) \quad \partial^2_{x \xi} \phi \text{ has maximal rank } n-1, \\
&H2^+) \quad \partial^2_{\xi \xi} \langle \partial_x \phi(x,\xi), G_0(x,\xi_0) \rangle \big\vert_{\xi = \xi_0} \text{ has } n-1 \text{ eigenvalues of the same sign.}
\end{split}
\end{equation*}
In the above display $G_0$ denotes the unnormalized Gauss map of the embedded surface $\xi \mapsto \partial_x \phi(x,\xi)$:
\begin{equation*}
G_0(x,\xi) = \partial^2_{x \xi_1} \phi(x,\xi) \wedge \partial^2_{x \xi_2} \phi(x,\xi) \ldots \wedge \partial^2_{x \xi_{n-1}} \phi(x,\xi) \in \Lambda^{n-1} \R^n \equiv \R^n.
\end{equation*}
$H1)$ and $H2^+)$ are sometimes referred to as Carleson--Sj\"olin condition.



\medskip

The analysis in \cite{LeeRyu2022} shows that the main contribution can be reduced to an oscillatory integral operator, which satisfies $H1)$ and $H2^+)$. This can be regarded as analog of the Carleson--Sj\"olin reduction in case of the Euclidean Bochner--Riesz means. Guth--Hickman--Iliopoulou \cite{GuthHickmanIliopoulou2019} showed the sharp results for phase functions satisfying the above conditions, in the sense that there are non-degenerate and elliptic phase functions, which do not allow for further improved estimates.

\medskip

However, already in 1991, Bourgain \cite{Bourgain1991} observed that a certain curvature conditionallows for improved estimates of the oscillatory integral operator beyond non-degeneracy and ellipticity assumptions, whereas failure of the condition leads to failure of the $L^p$-bounds suggested by the restriction conjecture. The additional curvature condition will be referred to as \emph{Bourgain condition} following Guo--Wang--Zhang \cite{GuoWangZhang2022}. Under this curvature condition Guo--Wang--Zhang \cite{GuoWangZhang2022} reported the currently widest range of $L^p$-estimates for oscillatory integral operators, going beyond the range proved by Guth--Hickman--Iliopoulou \cite{GuthHickmanIliopoulou2019} for $n \geq 3$. The original definition in \cite{Bourgain1991} used normal forms. We use the following characterization \cite[Theorem~2.1]{GuoWangZhang2022}:
\begin{definition}
Let $\phi \in C^\infty(\R^n;\R^{n-1})$ be a phase function, which satisfies $H1)$ and $H2)$. We say that $\phi$ satisfies Bourgain's condition at $(x_0,\xi_0) \in \R^n \times \R^{n-1}$ if and only if
\begin{equation*}
(( G_0 \cdot \nabla_x)^2 \partial^2_{\xi \xi} \phi(x_0,\xi_0)) \text{ is a scalar multiple of } ( G_0 \cdot \nabla_x) \partial^2_{\xi \xi} \phi(x_0,\xi_0).
\end{equation*}
\end{definition}


\medskip

%

In the following we show that the phase function obtained in \cite{LeeRyu2022} does \emph{not} satisfy the Bourgain condition. We repeat the decompositions from \cite{LeeRyu2022} for convenience.

\medskip

After the decompositions by Lee--Ryu \cite{LeeRyu2022} the local Hermite Bochner--Riesz estimates are reduced to $L^p$-estimates for the oscillatory integral (cf. \cite[Eq.~(2.67)]{LeeRyu2022}):
\begin{equation}
\label{eq:ReductionOscillatoryIntegralLeeRyu}
\big\| \int e^{i \lambda \rho \tilde{\mathcal{P}}_{\mathcal{H}}(0,x,y)} A(x,y) f(y) dy \big\|_{L^p(\R^d)} \lesssim (\lambda \rho)^{-\frac{d}{p}} \| f \|_{L^p(\R^d)}
\end{equation}
with $\text{supp}(A) \subseteq \{(x,y) \in B(0,\epsilon_0) \times B(0,\epsilon_0) : |x-y| \geq c \epsilon_0 \}$ for some $c \ll 1$.

\medskip

For $(x,y) \in \mathfrak{D}(c_0)$ we define $S_c$ and $S_*$ implicitly by
\begin{equation*}
\begin{split}
\cos S_c(x,y) &= \langle x, y \rangle + \sqrt{\mathcal{D}(x,y)}, \\
\cos S_*(x,y) &= \langle x, y \rangle - \sqrt{\mathcal{D}(x,y)}.
\end{split}
\end{equation*}
The symmetric phase function is given by
\begin{equation*}
\begin{split}
\Phi_{\mathcal{H}}(x,y) &= \frac{1}{2} \big( S_c(x,y) + \frac{(|x|^2+|y|^2) \cos S_c(x,y) - 2 x \cdot y}{\sin S_c(x,y)} \big) \\
&= \frac{1}{2} \big( S_c(x,y) - \cos S_*(x,y) \sin S_c(x,y) \big).
\end{split}
\end{equation*}

For $\Phi_{\mathcal{H}}$ we consider $(x,y) \in \text{supp}(\chi_{\ell}) \times \text{supp}(\chi_{\ell}')$. The smooth functions $\chi_{\ell}$, $\chi'_{\ell}$ are adapted to balls of radius $\sim 2^{-\ell} \sim \rho$, which are centered at $(x_0,y_0)$. More precisely,
\begin{equation}
\label{eq:SupportConditions}
\begin{split}
&\quad \text{supp}(\chi_\ell) \times \text{supp}(\chi_{\ell}') \subseteq \mathfrak{D}(c_0/2), \\
&\quad \text{supp}(\chi_\ell), \; \text{supp}(\chi_{\ell}') \subseteq B(x_0,2^{2-\ell}) \text{ for some } x_0 \in \mathfrak{D}(c_0), \\
&\quad 2^{-\ell-2} \leq \text{dist}(\text{supp}(\chi_{\ell}),\text{supp}(\chi_{\ell}')) \leq 2^{-\ell}, \\
&\quad |\partial_x^\alpha \chi_{\ell}|, |\partial_x^\alpha \chi_{\ell}'| \leq C_\alpha 2^{|\alpha| \ell}.
\end{split}
\end{equation}

We obtain a phase function $\tilde{\mathcal{P}}_{\mathcal{H}}$ with uniformly bounded derivatives by rescaling:
\begin{equation*}
\tilde{\mathcal{P}}_{\mathcal{H}}(0,x,y) = \rho^{-1} \Phi_{\mathcal{H}}(x_0 + \rho x, y_0 + \rho y).
\end{equation*}

\medskip

To facilitate description of derivatives of $\Phi_{\mathcal{H}}$, we recall more notations from \cite{LeeRyu2022}:
\begin{equation*}
\a(x,y)= \cos(S_c(x,y)) x - y, \quad \b(x,y) = x - \cos(S_c(x,y)) y.
\end{equation*}
To keep compatibility with notations from \cite{LeeRyu2022}, we denote the transpose of a vector $x \in \R^d$ by $x^t$.
For $(x,y) \in \text{supp}(\chi_{\ell}) \times \text{supp}(\chi_{\ell}')$ like in \eqref{eq:SupportConditions} we have
\begin{equation*}
|\a| , |\b|, |\a^t \b | \gtrsim \rho.
\end{equation*}

It turns out that $\a$ spans the kernel of $\partial^2_{xy} \Phi_{\mathcal{H}}$. The curvature properties of $z \mapsto \partial_x \Phi_{\mathcal{H}}(x,z)$ are encoded by
\begin{equation*}
\mathbf{M}(x,y) = \partial^2_{zz} \big\langle \partial_x \Phi_{\mathcal{H}}(x,z), \frac{\a(x,y)}{|\a(x,y)|} \big\rangle \big\vert_{z = y}.
\end{equation*}

In the following for a matrix $A \in \mathbb{R}^{d \times d}$ we denote the $(d-1) \times (d-1)$-submatrix by $A'' = ( A_{ij} )_{1 \leq i,j \leq d-1}$ and let $I_d = \text{diag}(1,\ldots,1) \in \R^{d \times d}$ denote the unit matrix. It holds:
\begin{lemma}[{\cite[Lemma~2.13]{LeeRyu2022}}]
\label{lem:CarlesonSjoelin}
Let $(x,y) \in \text{supp}(\chi_l) \times \text{supp}(\chi_l')$. Then the following is true:
\begin{itemize}
\item[(i)] The matrix $\partial^2_{xy} \Phi_{\mathcal{H}}(x,y)$ has rank $d-1$ and
\begin{equation*}
\partial^2_{xy} \Phi_{\mathcal{H}}(x,y) \a(x,y) = 0.
\end{equation*}
\item[(ii)] If $(x,y) \in \text{supp}(\chi_l) \times \text{supp}(\chi_l')$ satisfies
\begin{equation*}
\frac{\b(x,y)}{|\b(x,y)|} = \mathbf{e}_d,
\end{equation*}
then the submatrix $\mathbf{M}''(x,y) = \{ \mathbf{M}(x,y)_{i,j} \}_{1 \leq i,j \leq d-1}$ of $\mathbf{M}(x,y)$ has negative eigenvalues $\lambda_1,\ldots,\lambda_{d-1}$ such that
\begin{equation*}
- \lambda_i \sim |x-y|^2, \quad 1 \leq i \leq d-1.
\end{equation*}
\end{itemize}
\end{lemma}

$\mathbf{M}$ can be computed explicitly:
\begin{lemma}[{\cite[Lemma~2.14]{LeeRyu2022}}]
\label{lem:CurvatureDescription}
Let $(x,y) \in \mathfrak{D}(c_0)$ and \\ $\omega(x,y) = \sqrt{(1-|x|^2) \mathcal{D}(x,y)} \sin^4 S_c(x,y)$. Then we have
\begin{equation}
\label{eq:CurvatureDescription}
\mathbf{M}(x,y) = \frac{\a^t \b I_{d} - \a \b^t}{\omega(x,y) \a^t \b} \big( \b \a^t - \a^t \b I_{d \times d} \big).
\end{equation}
\end{lemma}
We remark that we have simplified the expression in \cite{LeeRyu2022} observing
\begin{equation*}
(\a^t \b I_d - \a \b^t) \a \a^t = 0.
\end{equation*}


Then the $L^p$-$L^p$-estimates \eqref{eq:ReductionOscillatoryIntegralLeeRyu} follow from a reduced phase function $\phi_{y_0}$, which is obtained from $\Phi_{\mathcal{H}}$ by freezing a suitable component:
Suppose in the following by rotation invariance that 
\begin{equation}
\label{eq:Conditionb}
\frac{\mathbf{b}(x_0,y_0)}{|\b(x_0,y_0)|} = \mathbf{e}_d.
\end{equation}

Define (cf. \cite[p.~27]{LeeRyu2022})
\begin{equation*}
\phi_{y_d}(x,\xi)= \tilde{\mathcal{P}}_{\mathcal{H}}(0,x,\xi,y_d),
\end{equation*}
which was shown in \cite{LeeRyu2022} to satisfy the Carleson--Sj\"olin condition. The direction proportional to $\b$ is distinguished because $\b^t \mathbf{M} = 0, \; \mathbf{M} \b = 0$.

We show the following:
\begin{proposition}[Generic~failure~of~Bourgain~condition]
Let the notations be like above and assume \eqref{eq:Conditionb}. For $x_0$ in a neighbourhood of the origin, $\phi_{y_d}$ satisfies the Bourgain condition at the origin if and only if there is $c \in \R$ such that $x_0 = c y_0$.
\end{proposition}
\begin{proof}
Let $\mathbf{a}_{\rho}(x,y) = \mathbf{a}(\rho x + x_0, \rho y + y_0)$ for $(x,\xi,y_d) \in B(0,\epsilon_0) \times B(0,\epsilon_0)$.

Then it follows from Lemma \ref{lem:CurvatureDescription} that
\begin{equation*}
\partial_{\zeta} \big\langle \partial_x \phi_{y_d}(x,\zeta), \frac{\a_{\rho}(x,\xi,y_d)}{|\a_{\rho}(x,\xi,y_d)|} \big\rangle \big\vert_{\zeta = \xi} = 0
\end{equation*}
and
\begin{equation*}
\partial_{\xi \xi}^2 \big\langle \partial_x \phi_{(y_0)_d} (x,\xi), \frac{\a(x_0,y_0)}{|\a(x_0,y_0)|} \big\rangle \big\vert_{(x,\xi) = (0,0)} = \rho^2 \mathbf{M}''(x_0,y_0).
\end{equation*}
Indeed, we have
\begin{equation*}
\a_{\rho}(x,\xi) \sim \partial^2_{x \xi_1} \phi_{(y_0)_d}(x,\xi) \wedge \partial^2_{x \xi_2} \phi_{(y_0)_d}(x,\xi) \wedge \ldots \wedge \partial^2_{x \xi_{d-1}} \phi_{(y_0)_d}(x,\xi).
\end{equation*}
So we obtain
\begin{equation*}
(G_0 \cdot \nabla_x) \partial^2_{\xi \xi} \phi_{(y_0)_d}(x,\xi) \sim \mathbf{M}''(x_0 + \rho x,y_0' + \rho \xi, (y_0)_d).
\end{equation*}

Set
\begin{equation}
\label{eq:TildeM}
\tilde{\mathbf{M}}(x,y) = (\a^t \b) \b \a^t - (\a^t \b)^2 I_{d} - (\b^t \b) \a \a^t + (\a^t \b) \a \b^t,
\end{equation}
which satisfies $\tilde{\mathbf{M}} \sim \mathbf{M}$.
By the above observation, the Bourgain condition holds for $\phi_{(y_0)_d}$ at the origin if and only if
\begin{equation}
\label{eq:BourgainConditionReduction}
(\a \cdot \nabla_x) \tilde{\mathbf{M}}''(x_0,y_0) \sim \tilde{\mathbf{M}}''(x_0,y_0).
\end{equation}
Here we use that it is admissible to omit scalar functions and that $G_0 \sim \a$. 

We introduce notations:
\begin{equation*}
\partial_a  = \a \cdot \nabla, \quad \partial_a \cos(S_c(x,y)) = \kappa(x,y).
\end{equation*}
We will check that
\begin{equation*}
\partial_a \tilde{\mathbf{M}}''(x_0,y_0) = \lambda(x_0,y_0) I_{d-1}, \quad \lambda(0,y_0) \neq 0, \quad \lambda \in C^\infty,
\end{equation*}
whereas
\begin{equation*}
\tilde{\mathbf{M}}''(x_0,y_0) = - (\a \a^t)'' - (a_d b_d)^2 I_{d-1}.
\end{equation*}
The following observation concludes the proof:
\begin{equation}
\label{eq:Observation}
\tilde{\mathbf{M}}''(x_0,y_0) \not \sim I_{d-1} \Leftrightarrow \a \not \sim \b \Leftrightarrow x_0 \not \sim y_0.
\end{equation}

We have
\begin{equation*}
\begin{split}
\begin{pmatrix}
\a \\ \b
\end{pmatrix}
&=
\begin{pmatrix}
\cos(S_c) & - 1\\
1 & -\cos(S_c)
\end{pmatrix}
\begin{pmatrix}
x \\ y
\end{pmatrix} \\
\Leftrightarrow
\begin{pmatrix}
x \\ y
\end{pmatrix}
&=
\frac{1}{1-\cos^2(S_c)}
\begin{pmatrix}
- \cos(S_c) & 1 \\
-1 & \cos(S_c)
\end{pmatrix}
\begin{pmatrix}
\a \\ \b
\end{pmatrix}
.
\end{split}
\end{equation*}
We obtain
\begin{equation}
\label{eq:CovariantDerivative}
\begin{split}
\partial_a \a &= \frac{\kappa(x,y)}{1-\cos^2(S_c(x,y))} \b \\
&\quad + \frac{- \kappa(x,y) \cos(S_c(x,y)) + (1-\cos^2(S_c(x,y)) \cos(S_c(x,y))  }{1-\cos^2(S_c(x,y))} \a \\
&= \beta \a + \alpha \b, \\
\partial_a \b &= (1+\alpha) \a - \frac{\kappa(x,y) \cos(S_c(x,y))}{1-\cos^2(S_c(x,y))} \b = (1+\alpha) \a + \gamma \b.
\end{split}
\end{equation}

\medskip

Next, we compute by the Leibniz rule and \eqref{eq:CovariantDerivative}:
\begin{equation}
\label{eq:CompBCI}
\begin{split}
\partial_a (\a^t \b \b \a^t) &= \alpha (\b^t \b)(\b \a^t) + (1+\alpha) (\a^t \a) (\b \a^t) + (\a^t \b) (1+\alpha) (\a \a^t) + \alpha (\a^t \b) (\b \b^t) \\
&\quad + (2 \beta+ 2 \gamma) \a^t \b \b \a^t.
\end{split}
\end{equation}
The second line stems from substituting a multiple of the quantity of which we take the derivative. Note that by $(2,2)$-homogeneity in $\a$ and $\b$ of all the expressions in the second line of \eqref{eq:TildeM} we always obtain a $(2 \beta + 2 \gamma)$ multiple of the original $(2,2)$-homogeneous expression.


We further compute
\begin{equation}
\label{eq:CompBCII}
\partial_a ((\a^t \b)^2) = 2 (\alpha \b^t \b + \a^t \a (1+\alpha)) \a^t \b + 2(\gamma + \beta) (\a^t \b)^2.
\end{equation}

Next,
\begin{equation}
\label{eq:CompBCIII}
\partial_a (\a (\b^t \b) \a^t) = \alpha \b (\b^t \b) \a^t + \a (\b^t \b) \alpha \b^t + 2(1+\alpha) (\a^t \b) \a \a^t + 2(\beta + \gamma) \a (\b^t \b) \a^t.
\end{equation}

Moreover,
\begin{equation}
\label{eq:CompBCIV}
\begin{split}
\partial_a (\a \b^t (\a^t \b)) &= \alpha \b \b^t (\a^t \b) + (1+\alpha) \a \a^t \a^t \b + \alpha (\a \b^t \b^t \b) + \a \b^t \a^t (1+\alpha) \a \\
&\quad + 2(\beta + \gamma) \a (\b^t \b) \a^t.
\end{split}
\end{equation}

We obtain subsuming \eqref{eq:CompBCI}-\eqref{eq:CompBCIV}:
\begin{equation*}
\begin{split}
\partial_a \tilde{\mathbf{M}} &= \alpha (\b^t \b)(\b \a^t) + (1+\alpha) (\a^t \a) (\b \a^t) + (\a^t \b) (1+\alpha) (\a \a^t) + \alpha (\a^t \b) (\b \b^t) \\
&\quad - 2 (\alpha \b^t \b + \a^t \a (1+\alpha)) \a^t \b \\
&\quad - \alpha \b (\b^t \b) \a^t - \a (\b^t \b) \alpha \b^t - 2(1+\alpha) (\a^t \b) \a \a^t \\
&\quad + \alpha \b \b^t (\a^t \b) + (1+\alpha) \a \a^t \a^t \b + \alpha (\a \b^t \b^t \b) + \a \b^t \a^t (1+\alpha) \a.
\end{split}
\end{equation*}
Now we evaluate at $(x_0,y_0)$ and take advantage of $\b \sim \mathbf{e}_d$, simplifying matters:
\begin{equation}
\label{eq:DirectionalDerivativeM''}
\partial_a \tilde{\mathbf{M}}''(x_0,y_0) = -2 (\alpha \b^t \b + \a^t \a (1+\alpha)) \a^t \b I_{d-1} = - (2 \partial_a (\a^t \b)) \a^t \b I_{d-1}.
\end{equation}
To show that the expression does not vanish for our choice of $(x_0,y_0)$, we use the identity \cite[Eq.~(2.56)]{LeeRyu2022}:
\begin{equation*}
\a^t \b = \sqrt{\mathcal{D}(x,y)} \sin^2(S_c(x,y)) = \sqrt{\mathcal{D}(x,y)} (1-\cos^2(S_c(x,y)).
\end{equation*}
It follows
\begin{equation*}
\partial_a (\a^t \b) = \frac{\partial_a \mathcal{D}}{2 \sqrt{\mathcal{D}(x,y)}} \sin^2(S_c(x,y)) - \sqrt{\mathcal{D}(x,y)} \cos(S_c(x,y)) \partial_a \cos(S_c(x,y)).
\end{equation*}
We have
\begin{equation*}
\begin{split}
\partial_a \mathcal{D}(x,y) &= 2 (\a \cdot y) (x \cdot y) - 2 \a \cdot x, \\
\partial_a \cos(S_c(x,y)) &= \a \cdot y + \frac{\partial_a \mathcal{D}(x,y)}{2 \sqrt{\mathcal{D}(x,y)}}.
\end{split}
\end{equation*}
This implies $\partial_a \mathcal{D}(0,y_0) = 0$ and
\begin{equation*}
\big( \partial_a \cos(S_c(x,y)) \big) \vert_{(x,y) = (0,y_0)} = - \| y \|^2.
\end{equation*}
Since $|\a^t \b| \sim |y|^2$, we can conclude
\begin{equation*}
\partial_a \tilde{\mathbf{M}}''(x_0,y_0) = \lambda(x_0,y_0) 1_d, \quad \lambda(0,y_0) \sim |y_0|^4.
\end{equation*}
Smoothness of $\lambda$ is clear and shows that $\partial_a \tilde{\mathbf{M}}''$ is not vanishing.
The proof is complete by \eqref{eq:Observation}.
\end{proof}

\section*{Acknowledgements}
Part of this work was conducted at the Korea Institute for Advanced Study, whose financial support through the grant No. MG093901 is gratefully acknowledged.
The paper was finished while visiting UC Berkeley, and I would like to thank Daniel Tataru for kind hospitality.

\end{document}